\documentclass[11pt]{article} 

\usepackage{fullpage} 
\usepackage{float}
\usepackage{authblk}
\usepackage{pb-diagram} 
\usepackage{url}
\usepackage{hyperref}
\usepackage{algorithm} 
\usepackage{algpseudocode} 
\usepackage{multirow} 
\usepackage{hhline}  
\usepackage{subfigure}
\usepackage{amsmath}
\usepackage{xcolor}
\usepackage{soul}
\usepackage{graphicx}
\usepackage{amscd}
\usepackage{appendix}
\usepackage{amssymb, mathrsfs}
\input{epsf.sty}
\usepackage{amsthm, amscd}
\usepackage{color}
\usepackage{bbm}
\usepackage{latexsym}
\usepackage{epic}
\usepackage{appendix}
\usepackage{enumerate}
\usepackage{longtable}
\usepackage{lscape}
\usepackage{extarrows}
\usepackage{blkarray}
\usepackage{multirow}
\usepackage{listings}
\usepackage{verbatim}
\newlength\myindent

\newcommand\smallO{
  \mathchoice
    {{\scriptstyle\mathcal{O}}} 
    {{\scriptstyle\mathcal{O}}} 
    {{\scriptscriptstyle\mathcal{O}}}
    {\scalebox{.6}{$\scriptscriptstyle\mathcal{O}$}}
  }

\usepackage{titlesec}
\titleformat{\section}
  {\normalfont\fontsize{15}{15}\bfseries}{\thesection}{1em}{}
  
\makeatletter
\def\input@path{{./}{./pstex/}{./image/}}
\graphicspath{{./}{./pstex/}{./image/}}
\makeatother

\newcommand{\comm}[1]{}

\newcommand{\pacomm}[1]{{}}

\newcommand{\delete}[1]{}

\newcommand{\Rmnum}[1]{\expandafter\@slowromancap\romannumeral #1@}
\newcommand{\inner}[3][]{{\langle #2, #3 \rangle_{#1}}}

\usepackage{amsthm}

\theoremstyle{plain}
\newtheorem{theorem}{Theorem}[section]
\newtheorem{lemma}[theorem]{Lemma}
\newtheorem{proposition}[theorem]{Proposition}
\newtheorem{corollary}[theorem]{Corollary}

\theoremstyle{definition}
\newtheorem{definition}[theorem]{Definition}
\newtheorem{assumption}[theorem]{Assumption}

\theoremstyle{remark}
\newtheorem{remark}[theorem]{Remark}

\numberwithin{equation}{section}

\DeclareMathOperator{\T}{\mathrm{T}}
\DeclareMathOperator{\F}{\mathrm{F}}
\DeclareMathOperator{\R}{\mathbb{R}}
\DeclareMathOperator{\Hess}{\mathrm{Hess}}
\DeclareMathOperator{\Proj}{\mathrm{Proj}}
\DeclareMathOperator{\grad}{\mathrm{grad}}
\DeclareMathOperator{\prox}{\mathrm{prox}}

\DeclareMathOperator{\I}{\mathrm{I}}
\DeclareMathOperator{\N}{\mathrm{N}}
\DeclareMathOperator{\D}{\mathrm{D}}

\DeclareMathOperator{\trace}{\mathrm{trace}}

\DeclareMathOperator{\sign}{\mathrm{sign}}

\DeclareMathOperator{\diag}{\mathrm{diag}}

\DeclareMathOperator*{\argmin}{argmin}

\DeclareMathOperator{\sgn}{\mathrm{sgn}}
\DeclareMathOperator{\M}{\mathcal{M}}

\DeclareMathOperator{\W}{\mathcal{W}}
\DeclareMathOperator{\U}{\mathcal{U}}
\DeclareMathOperator{\J}{\mathcal{J}}
\DeclareMathOperator{\K}{\mathcal{K}}

\begin{document}

\title{A Riemannian Proximal Newton Method
\footnotetext{Corresponding author: 
Wen Huang (\texttt{wen.huang@xmu.edu.cn}). WH was partially supported by the National Natural Science Foundation of China (NO. 12001455). Wutao Si was partially supported by the National Natural Science Foundation of China (NO. 12001455, No. 12171403) and the Fundamental Research Funds for the Central Universities (No. 20720210032). P.-A. Absil was supported by the Fonds de la Recherche Scientifique -- FNRS and the Fonds Wetenschappelijk Onderzoek -- Vlaanderen under EOS Project no 30468160, and by the Fonds de la Recherche Scientifique -- FNRS under Grant no T.0001.23. Simon Vary is a beneficiary of the FSR Incoming Post-doctoral Fellowship.}
}

\author[1]{Wutao Si}
\author[2]{P.-A. Absil}
\author[1]{Wen Huang}
\author[3]{Rujun Jiang}
\author[2]{Simon Vary}

\affil[1]{School of Mathematical Sciences, Xiamen University, Xiamen, China.\vspace{.15cm}}
\affil[2]{ICTEAM Institute, UCLouvain, Louvain-la-Neuve, Belgium.\vspace{.15cm}}
\affil[3]{School of Data Science, Fudan University, Shanghai, China.}

\maketitle




\begin{abstract}
In recent years, the proximal gradient method and its variants have been generalized to Riemannian manifolds for solving optimization problems with an additively separable structure, i.e.,  $f + h$, where $f$ is continuously differentiable, and $h$ may be nonsmooth but convex with computationally reasonable proximal mapping. In this paper, we generalize the proximal Newton method to embedded submanifolds for solving the type of problem with $h(x) = \mu \|x\|_1$. 
The generalization relies on the Weingarten and semismooth analysis. It is shown that the Riemannian proximal Newton method has a local quadratic convergence rate under certain reasonable assumptions. Moreover, a hybrid version is given by concatenating a Riemannian proximal gradient method and the Riemannian proximal Newton method. It is shown that if  the switch parameter is chosen appropriately, then the hybrid method converges globally and also has a local quadratic convergence rate. Numerical experiments on random and synthetic data are used to demonstrate the performance of the proposed methods.
\end{abstract}



\section{Introduction}\label{sec:Intro}
In this paper, we consider the following optimization problem
\begin{equation}\label{1-1}
\begin{aligned}
\min_{x \in \M} F(x) = f(x) + h(x)\ \hbox{with}\ h(x) = \mu \|x\|_1,
\end{aligned}
\end{equation} 
where $\M $ is a finite-dimensional 
embedded submanifold of an Euclidean space, see Definition \ref{def-2-1},
$f$ is a sufficiently smooth function, and $\mu > 0$.
This optimization problem arises in many important applications, such as sparse principal component analysis \cite{Zou2006,Zou2018}, sparse partial least squares regression~\cite{Chen2019}, compressed model~\cite{Ozolin2013}, sparse inverse covariance estimation~\cite{Bollhofer2019},  sparse blind deconvolution~\cite{Zhang2017} and clustering~\cite{Huang2022,Lu2016,Park2018}.

In the case that the manifold constraint is dropped, i.e., $\M = \R^n$, and the function $h$ is not restricted to be $\mu \|x\|_1$ but a continuous, convex, and not necessarily smooth function, the Euclidean nonsmooth  problem \eqref{1-1} has been extensively studied 
in~\cite{Beck2017,Nesterov2018}. The well-known methods include the proximal gradient method and its variants, which have found many practical successes in applications~\cite{Beck2009,Shi2014,Tibshirani1996,Widmer2012}. The proximal gradient method updates the iterate via
\begin{equation}\label{1-2}
\begin{cases}v_k = \argmin_{v \in \R^n}\ f(x_k) + \nabla f(x_k)^{\T} v + \frac{1}{2 t} \|v\|_{\F}^2 + h(x_k + v), & \text{(Proximal mapping)} \\ x_{k+1}=x_k + v_k, & \text{(Update iterates)}\end{cases}
\end{equation}
where the proximal mapping in
\eqref{1-2}  uses the first-order approximation of $f$ around the current estimate $x_k$.  If the function $f$ is convex, then under certain standard assumptions,
the proximal gradient method has an  $\mathcal{O}(1\slash k)$ rate of convergence~\cite[Chapter~10]{Beck2017}. 
Moreover, multiple accelerated versions of the proximal gradient method have been proposed in~\cite{Attouch2016,Beck2009b,Li2015,Villa2013}, which achieve the optimal convergence rate $\mathcal{O}(1/k^2)$~\cite{nesterov1983method}.
If the function $f$ is further assumed to be strongly convex, then the convergence rate of the proximal gradient method can be shown to be linear.

With the presence of the manifold constraint, the nonsmooth optimization problem \eqref{1-1} becomes more challenging. The update iterates in \eqref{1-2} can be generalized to the Riemannian setting using a standard technique called \emph{retraction}. However, generalizing the proximal mapping in \eqref{1-2} to the Riemannian setting is not straightforward, and multiple versions have been proposed. In \cite{2020Proximal}, a proximal gradient method called ManPG for solving the optimization problem over the Stiefel manifold is proposed and the update parallel to \eqref{1-2} is given by
\begin{equation}\label{1-4}
\hspace{-0.4em}
\begin{cases} v_k = \argmin_{v \in \T_{x_k} \M}\ f(x_k) + \nabla f(x_k)^{\T} v + \frac{1}{2t} \|v\|_{\F}^2 + h(x_k + v), & \hspace{-1em} \text{(Proximal mapping)} \\ x_{k+1}=R_{x_k} (v_k). & \hspace{-0.8em} \text{(Update iterates)}\end{cases}
\end{equation}
Compared to the Euclidean setting, the Riemannian proximal mapping in~\eqref{1-4} has an additional linear constraint that ensures the search direction $v_k$ stays in the tangent space $\T_{x_k} \M$. It is shown in~\cite[Section~4.2]{2020Proximal} that the proximal mapping in~\eqref{1-4} can be solved efficiently by a semismooth Newton method.
Moreover, the global convergence of ManPG has been established. In~\cite{Huang2022a}, a diagonal weighted proximal mapping is defined by replacing $\|v\|_{\F}^2$ in \eqref{1-4} by $\langle v, W v \rangle$, where the diagonal weighted linear operator $W$ is motivated from the Hessian. In addition, an accelerated proximal gradient method is generalized to the Riemannian setting called AManPG that empirically exhibits accelerated behavior over the Stiefel manifold. The Riemannian generalization of the proximal mapping in~\eqref{1-4} requires being able to perform a linear combination of $x_k + v$,
 which cannot be defined on a generic manifold.  In~\cite{Huang2022b},  a Riemannian gradient method called RPG is proposed by replacing the addition $x_k + v$ with a retraction $R_{x_k}(v)$, which yields a Riemannian proximal mapping
\begin{equation}\label{1-5}
    v_k = \argmin_{v \in \T_{x_k} \M}\ f(x_k) + \langle \grad f(x_k), v\rangle_{x_k}
    + \frac{1}{2t} \|v\|_{x_k}^2 + h\left(R_{x_k}(v)\right),
\end{equation}
where $\grad f$ denotes the Riemannian gradient of $f$ and $\langle \cdot,\cdot\rangle_x$ is Riemannian metric defined on the tangent space $\T_x \M$. Not only the global convergence but also the local convergence rate has been established in terms of the Riemannian KL property. The same authors further propose the inexact Riemannian proximal gradient method called IRPG without solving \eqref{1-5} exactly in~\cite{Huang2023}. The global convergence and local convergence rates have also been given. Moreover, the search direction given by~\eqref{1-4} can be viewed as an inexact solution of the Riemannian proximal mapping~\eqref{1-5} that still guarantees global convergence. Though the Riemannian proximal gradient method with~\eqref{1-5} has nice global and local convergence results, it is still unknown whether Subproblem~\eqref{1-5} can be solved efficiently in general.

If $\M = \R^n$ and the function $f$ is twice continuously differentiable and strongly convex,  proximal Newton-type methods are proposed in~\cite{Lee2014} and achieve a superlinear convergence rate.
The proximal mapping in the proximal Newton-type methods replaces $f$ with a second-order approximation and yields the search direction
\begin{equation}\label{1-3}
    v_k = \argmin_{v \in \R^n}\ f(x_k) + \nabla f(x_k)^{\T} v + \frac{1}{2 } v^{\T} H_k v + h(x_k + v),
\end{equation}
where $H_k$ represents a suitable approximation of the exact Hessian $\nabla^2 f(x_k)$. If $H_k$ is chosen to be the Hessian, i.e., $H_k = \nabla^2  f(x_k)$, then the search direction~\eqref{1-3} yields the proximal Newton method.  The Euclidean proximal Newton-type method traces its prototype back to~\cite{josnewton,josquasi}, where it was primarily used to solve generalized equations. Over the next few decades, this method was extensively studied and has been proven to be efficient and effective in many applications~\cite{Karimi2017,Pan2013,Santos2018,NIPS2014}. In the Euclidean setting, if the function $f$ is twice continuously differentiable and strongly convex, the Hessian of $f$ is Lipschitz continuous, and the function $h$ is convex, then the proximal Newton method converges quadratically. When generalizing to the Riemannian setting, some of the assumptions may be too strong. For example, a geodesically-convex function, which is a commonly-used Riemannian generalization of convexity, on a compact manifold must be a constant function, see~\cite[Corollary~11.10]{boumal2023}.  In~\cite{Mordukhovich2023}, the authors remove the global convexity of $f$ of the Euclidean proximal Newton method while still guaranteeing global convergence and local superlinear convergence. However, the manifold curvature appears in the second-order information of the objective function, which still significantly increases the difficulty of generalizing to the Riemannian setting. The recent paper~\cite{Wang2023} proposes the proximal quasi-Newton method, called ManPQN, over the Stiefel manifold with the Riemannian proximal mapping
\begin{equation}\label{manpqn}
    v_k = \argmin_{v \in \T_{x_k} \M}\ f(x_k) + \langle \grad f(x_k), v\rangle
    + \frac{1}{2} \|v\|_{\mathcal{B}_k}^2 + h\left(x_k + v\right),
\end{equation}
where $\mathcal{B}_k$ is a symmetric positive definite operator on $\T_{x_k}\M$ and $\|v\|_{\mathcal{B}_k}^2 = \langle v, \mathcal{B}_k[v]\rangle$. Additionally, the global convergence of ManPQN has been established. However, in both theoretical and numerical results, only the local linear convergence of ManPQN has been demonstrated. Note that the naive generalization of replacing the Euclidean gradient and Hessian with the Riemannian counterparts generally does not yield a superlinear convergence result, see details in Subsection~\ref{subsec:naive}.

The main contributions of this paper are summarized as follows. A Riemannian proximal Newton method, called RPN, is proposed and studied. This method is based on the idea of semismooth implicit function analysis, which is different from the naive generalization of Euclidean setting~\cite{Lee2014}. It is proven that the proposed algorithm is capable of achieving quadratic convergence under certain reasonable assumptions. The local result requires that the iterative point is sufficiently close to an optimal point. We show that the distance between the iterative point and the optimal point can be controlled by the norm of search direction. Therefore, a hybrid version by concatenating a Riemannian proximal gradient method and the Riemannian proximal Newton method is given, and its global and local quadratic convergence is guaranteed. Numerical experiments are used to demonstrate the performance of the proposed methods.

There also exist multiple generic nonsmooth optimization algorithms on Riemannian manifolds that can be used to solve Problem~\eqref{1-1}, such as subgradient-based algorithms~\cite{GH2016a,li2021weakly}, 
Riemannian gradient sampling algorithm~\cite{HU2017}, 
and Riemannian proximal bundle algorithm~\cite{HNP2021}. These algorithms do not exploit the structure of the objective function in~\eqref{1-1}.
In~\cite{ZBDZ2022}, an augmented Lagrangian method that uses the structure of~\eqref{1-1} is proposed, where the subproblem needs to find a zero of a semismooth vector field on a matrix manifold. A semismooth Newton method is developed therein for the subproblem and is proven to converge superlinearly under certain conditions. The convergence rate of the augmented Lagrangian method is given later in~\cite{ZBD2022} and is shown to converge linearly.

This paper is organized as follows. Notation and preliminaries are given in Section~\ref{sec:NotPre}. The Riemannian proximal Newton method together with its local convergence analysis are presented in Section~\ref{sec:RPN}. The hybrid algorithm is described and analyzed in Section~\ref{sec:glo}.
Numerical experiments are reported in Section~\ref{sec:Num}. Finally, we draw some concluding remarks and potential future directions in Section~\ref{sec:con}.

\section{Notation and Preliminaries}\label{sec:NotPre}

\subsection{Preliminaries on Riemannian Submanifolds of Euclidean Spaces}\label{subsec:Rie}

An $n$-dimensional Euclidean space is denoted by $\R^n$. The Euclidean space $\R^n$ does not only refer to a vector space, but also can refer to a matrix space or a tensor space. In a Euclidean space, the Euclidean metric is typically represented by $\langle u, v \rangle$, which is defined as the sum of the entry-wise products of $u$ and $v$, such as $u^{\T}v$ for vectors and $\trace(u^{\T}v)$ for matrices. The Frobenius norm is denoted by $\|\cdot\|_{\F} = \sqrt{\inner[]{\cdot}{\cdot}}$. A linear operator $\mathcal{A}$ on the Euclidean space $\R^n$ is call self-adjoint or symmetric if it satisfies $\inner[]{\mathcal{A} x}{y} = \inner[]{x}{\mathcal{A} y}$, for all $x, y \in \R^n$. If a symmetric linear operator $\mathcal{A}$ satisfies $\inner[]{\mathcal{A} x}{x} > (\geq) 0$ for all $x \in \R^n \setminus \{0\}$, then $\mathcal{A}$ is called a symmetric positive definite (semidefinite) linear operator and denoted by $\mathcal{A} \succ (\succeq) 0$. 
For $x\in \R^n$, $\|x\|_1$ denotes the sum of the absolute values of all entries in $x$ and $\sgn(x) \in \R^n$ denotes the sign function, i.e., $(\sgn(x))_i = -1$ if $x_i < 0$; $(\sgn(x))_i = 0$ if $x_i = 0$; and $(\sgn(x))_i = 1$ otherwise. For $x, y \in \mathbb{R}^n$, $x \odot y$ denotes $z \in \mathbb{R}^n$ such that $z_i = x_i y_i$ for all $i$, that is, $\odot$ denotes the Hadamard product.

The following is the standard definition of an embedded submanifold~\cite{absil2009,boumal2023}, which is used in the proof of Lemma~\ref{lemm-3-5}. Roughly speaking, an embedded submanifold in an Euclidean space is either an open subset or a smooth surface in the space.

\begin{definition}[Embedded submanifolds of $\R^n$~\cite{boumal2023} ]\label{def-2-1}
Let $\M$ be a subset of a Euclidean space $\R^n$. We say~$\M$ is a (smooth) embedded submanifold of $\R^n$ if either $\mathcal{M}$ is an open subset of $\R^n$ or for a fixed integer $k \ge 1$ and for each $x\in \M$ there exists a neighbourhood $\mathcal{U}$ of $x$ in $\R^n$ and a smooth function $h: \mathcal{U} \to \R^k$ such that
\begin{itemize}
    \item[(a)] If $y$ is in $\mathcal{U}$, then $h(y)=0$ if and only if $y\in \mathcal{M}$; and
    \item[(b)] $\operatorname{rank} \mathrm{D}h(x)=k$,
\end{itemize}
where $\mathrm{D}$ denotes the differential operator. Such a function $h$ is called a local defining function for $\M$ at $x$.
\end{definition}

The tangent space of the manifold $\M$ at $x$ is denoted by $\T_x \M$,
and the tangent bundle, denoted $\T \M$, is the set of all tangent vectors. Since the tangent space is a vector space, one can equip it with an inner product (or metric) $\langle \cdot,\cdot\rangle_x: \T_x\M \times \T_x \M \to \R$. The induced norm from the metric is denoted by $\|\cdot\|_x = \sqrt{ \inner[x]{\cdot}{\cdot}}$. The subscript $x$ of $\|\cdot\|_x$ is omitted if it is clear in the context.
A manifold whose tangent spaces are endowed with a smoothly varying metric is referred to as a Riemannian manifold. 
If the manifold $\M$ is an embedded submanifold of $\R^n$ and the Riemannian metric of $\M$ is endowed from $\R^n$, then $\M$ is called a Riemannian submanifold of $\R^n$. Throughout this paper, the manifold $\mathcal{M}$ is assumed to be a Riemannian submanifold of $\mathbb{R}^n$.

When the manifold $\M$ is an embedded submanifold of $\R^n$, the tangent space $\T_x \M$ is a linear subspace of $\R^n$. In this paper, we assume that the dimension of $\T_x \mathcal{M}$ is $n-d$.  One can define the orthogonal complement space of $\T_x \M$ as $\N_x \M = \T_x^{\perp} \M$, called the normal space of $\M$ at $x$.  If one can find a basis of $\N_x \M$, denoted by $B_x = \left(b_x^{(1)}, b_x^{(2)}, \ldots, b_x^{(d)} \right)$, then the tangent space $\T_x \M$ can be characterized as its orthogonal complement
\[
\T_x \M = \left\{v \in \mathbb{R}^n: B_x^T v := \left( \inner[]{b_x^{(1)}}{v}, \inner[]{b_x^{(2)}}{v}, \ldots, \inner[]{b_x^{(d)}}{v} \right)^T = 0\right\}.
\]
By~\cite[Page 139]{huang2013optimization} and~\cite{Huang_Absil2017}, the map $x\mapsto Q_x$ can be chosen to be smooth at least locally, where $Q_x$ is an orthonormal basis of $\T_x \M$. Since $\N_x \M = \T_x^{\perp} \M$,  a smooth orthonormal basis $B_x$ can be obtained, i.e. the map $x\mapsto B_x$ is smooth in a neighborhood of $x$ for any $x \in \M$. We will see in Section~\ref{sec:RPN}, that such a characterization of $\T_x \M$ is useful in reformulating Riemannian proximal mappings. 

The Riemannian gradient of a smooth function $f: \M \to \R$ at $x$, denoted by $\grad f(x)$, is the unique tangent vector satisfying
\begin{equation} \label{eq:04}
\mathrm{D} f(x)[\eta_x] = \left\langle \eta_x, \grad f(x)\right\rangle_x,\ \forall \eta_x\in \T_x \M,
\end{equation}
where $\langle \cdot, \cdot \rangle_x$ is the Riemannian metric on $\T_x \M$ and $\mathrm{D} f(x)[\eta_x]$ is the directional derivative of $f$ at $x$ along $\eta_x$. If $\M$ is a Riemannian submanifold of $\R^n$, then Riemannian gradient $\grad f(x)$ has the following explicit statement, 
\[
\grad f(x) =  \Proj_x \left(\nabla f(x) \right),
\]
where $\nabla f(x)$ is the Euclidean gradient, $\Proj_x$ denotes the orthogonal projector onto $\T_x \M$, and the orthogonality is defined by the Euclidean metric. 

The Riemannian Hessian of $f$ at $x$, denoted by $\Hess f(x)$, is a linear operator on $\T_x \M$ satisfying
\[
\Hess f(x)[\eta_x] = \Bar{\nabla}_{\eta_x} \grad f(x),\ \forall \eta_x \in \T_x \M ,
\]
where $\Hess f(x)[\eta_x]$ denotes the action of $\Hess f(x)$ on a tangent vector $\eta_x$, and $\Bar{\nabla}$ denotes the Riemannian affine connection, see~\cite[Section~5.3]{absil2009} and~\cite[Section~5.4]{
boumal2023}. Roughly speaking, an affine connection generalizes the concept of a directional derivative of a vector field. Furthermore, if $\M$ is a Riemannian submanifold of $\R^n$, then the action of Riemannian Hessian $\Hess f(x)$ along the direction $\eta_x \in \T_x \M$ has the following explicit expression, see~\cite[Section~5.5 and~5.11]{
boumal2023}, i.e.,
\begin{equation}\label{eq:Rhess}
    \begin{aligned}
    \Hess f(x)[\eta_x] &= \Proj_x\left( \mathrm{D} \grad f(x) [\eta_x] \right)\\
    &= \Proj_x \left(\nabla^2 f(x) [\eta_x]\right) + \W_x\left(\eta_x, \Proj_x^{\perp} \big(\nabla f(x)\big) \right), 
    \end{aligned}
\end{equation}
where $\nabla^2 f(x)$ is the Euclidean Hessian, $\Proj_x^{\perp} = \mathrm{Id} - \Proj_x$ is the orthogonal projector to the normal space $\N_x \M = (\T_x \M)^{\perp}$, $\mathrm{Id}$ denotes the identity operator, $\W_x$ is the Weingarten map~\cite[Section~5.11]{boumal2023} defined by
\begin{equation}\label{2-3}
    \W_x:  \T_x \M \times  \N_{x} \M \to  \T_x \M:(w,u)\mapsto \W_x(w,u) = \D(x\mapsto \Proj_x)(x)[w] \cdot u, 
\end{equation}
where $\D(x\mapsto \Proj_x)(x)$ is the differential of the function $\Proj_x$ at $x$. The Weingarten map is related to the curvature of the manifold, see~\cite[Chapter~8]{lee2018introduction}.

A \emph{retraction} on a manifold $\M$ is a smooth mapping from the tangent bundle $\T\M$ to $\M$ such that  $(i) R_x (0_x) = x$, where $0_x$ is the zero vector in $\T_x \M$; $(ii)$ the differential of $R_x$ at $0_x$, denoted $\mathrm{D}R_x(0_x)$, is the identity map. Although the domain of a retraction does not necessarily need to be the whole tangent bundle, it is often the case in practice. Retractions whose domain is the whole tangent bundle are referred to as \emph{globally defined} retractions.
We denote $R_x : \T_x \M \to \M$ to be the restriction of $R$ to $\T_xM$. For any $x\in \M$, there always exists a neighborhood of $0_x$ such that $R_x$ is a diffeomorphism in the neighborhood.


\subsection{Preliminaries on Implicit Function Theorems}

The proposed proximal Newton method relies on an implicit function theorem for semismooth functions. In this section, the implicit function theorems for both smooth and semismooth functions are reviewed. We refer interested readers to \cite{dontchev2009,Gowda2004,krantz2002,Pang2003,Sun2001} for more details. Lemma \ref{lemm-2-1} states the well-known implicit function theorem for continuously differentiable functions.

\begin{lemma}[Implicit Function Theorem]\label{lemm-2-1}
Let $F:\mathbb{R}^{n+m} \to \mathbb{R}^{m}$ be a continuously differentiable (i.e., $\mathbb{C}^1$) function, and $F(x^0,y^0) = 0$. If the Jacobian matrix $J_y F(x^0,y^0)$ is invertible, then there exists an open set $U \subset \mathbb{R}^n$ containing $x^0$ such that there exists a unique $\mathbb{C}^1$ function $f: U\to \mathbb{R}^m$ such that $f(x^0) = y^0$, and $F(x, f(x)) = 0$ for all $x\in U$. Moreover, the Jacobian matrix of partial derivatives of $f$ in $U$ is given by the matrix product
$
Jf(x) = - [J_y F(x,y)]^{-1} J_x F(x,y).
$ 
\end{lemma}

However, in many applications, we encounter functions that are not differentiable everywhere, which leads to the following definition of semismoothness.

\begin{definition}[Semismoothness~\cite{Li2018,Qi1993}]\label{def-2-3}
Let $\mathcal{D} \subset \R^n$ be an open set, $\mathcal{K}: \mathcal{D} \rightrightarrows \R^{m\times n}$ be a nonempty and compact valued, upper semicontinuous set-valued mapping, and let  $F:\mathcal{D} \to \mathbb{R}^m$ be a locally Lipschitz continuous function. We say that $F$ is semismooth at $x\in \mathcal{D}$ with respect to $\mathcal{K}$ if
\begin{itemize}
    \item[(a)] $F$ is directionally differentiable at $x$; and
    \item[(b)] for any $J \in \mathcal{K}(x+d)$,
    \[
    F(x+d) - F(x) - J d = \smallO(\|d\|) \ \mathrm{as}\ d\to 0.
    \]
\end{itemize}
Furthermore, $F$ is said to be strongly semismooth at $x\in \mathcal{D}$ with respect to $\mathcal{K}$ if $F$ is semismooth at $x\in \mathcal{D}$ and for any $J \in \mathcal{K}(x+d)$,
    \[
    F(x+d) - F(x) - J d = \mathcal{O}(\|d\|^2) \ \mathrm{as}\ d\to 0.
    \]
If $F$ is (strongly) semismooth at any $x \in \mathcal{D}$ with respect to $\mathcal{K}$, then $F$ is called a (strongly) semismooth function with respect to $\mathcal{K}$.
\end{definition}

Commonly-encountered semismooth functions include smooth functions, piecewise smooth functions, and convex functions. For every $p \in [1, \infty]$, the norm $\|\cdot\|_p$ is strongly semismooth. Piecewise affine functions are strongly semismooth. Note that the proximal mapping  of $\ell_1$ norm $\|\cdot\|_1$, i.e., $\prox_{\|\cdot\|_1}$ is component-wise separable and piecewise affine, then $\prox_{\|\cdot\|_1}$ is strongly semismooth.
It has been shown in~\cite{Qi1993} that the corresponding set-valued mapping $\mathcal{K}$ can be chosen as the B(ouligand)-subdifferential or the Clarke subdifferential, i.e.,
\[
\mathcal{K}:\mathcal{D} \mapsto \mathbb{R}^{m \times n}: x \mapsto \partial_{\mathrm{B}} F(x) \quad \hbox{ or } \quad \mathcal{K}:\mathcal{D} \mapsto \mathbb{R}^{m \times n}: x \mapsto \partial F(x),
\]
where $\partial_{\mathrm{B}} F(x)$ and $\partial F(x)$ respectively denote the B(ouligand)-subdifferential and the Clarke's subdifferential, and the definitions of B(ouligand)-subdifferential and the Clarke's subdifferential can be found in \cite{clarke1990} and we state them in Definition~\ref{def:B-C-subdiff} for completeness.

\begin{definition}[B-subdifferential and Clarke's subdifferential] \label{def:B-C-subdiff}
Let $F:\mathcal{D} \subset \mathbb{R}^n \rightarrow \mathbb{R}^m$ be a locally Lipschitz continuous function, where $\mathcal{D}$ is an open subset of $\mathbb{R}^n$. Let $\mathcal{D}_F$ denote the set of the points in $\mathcal{D}$ such that $F$ is differentiable at any $x \in \mathcal{D}_F$\footnote{Note that by Rademacher's theorem, any locally Lipschitz continuous function is differentiable almost everywhere in its domain. In other words, $\mathcal{D}_F$ is a dense subset of $\mathcal{D}$ in the sense of the Lebesgue measurement.}. The B-subdifferential of $F$ at $x \in \mathcal{D}$ is defined by
\[
\partial_{\mathrm{B}} F(x) = \left\{\lim_{k\to +\infty} J F(x_k): x_k \in \mathcal{D}_F, \lim_{k\to +\infty} x_k  = x\right\}.
\]
The Clarke's subdifferential of $F$ at $x \in \mathcal{D}$ is defined as the convex hull of $\partial_{\mathrm{B}} F(x)$, i.e., $\partial F(x) = \mathrm{conv}\{\partial_{\mathrm{B}} F(x)\}$.
\end{definition}

In~\cite[Theorem~4]{Gowda2004}, the author proposed an implicit function theorem for semismooth function, however, the semismoothness in~\cite{Gowda2004} is weaker than the semismoothness in Definition~\ref{def-2-3}. To distinguish this kind of semismoothness, the authors in~\cite{Pang2003} called semismoothness in~\cite{Gowda2004} to be G-semismoothness, where it is defined as follows. 

\begin{definition}[G-semismoothness~\cite{Gowda2004, Pang2003}]
    Let $F:\mathcal{D} \to \mathbb{R}^m$ where $\mathcal{D} \subset \R^n$ be an open set, $\mathcal{K}: \mathcal{D} \rightrightarrows \R^{m\times n}$ be a nonempty set-valued mapping. We say that $F$ is G-semismooth at $x\in \mathcal{D}$ with respect to $\mathcal{K}$ if for any $J \in \mathcal{K}(x+d)$,
    \[
    F(x+d) - F(x) - J d = \smallO(\|d\|) \ \mathrm{as}\ d\to 0.
    \]
    Furthermore, $F$ is said to be G-strongly semismooth at $x\in \mathcal{D}$ with respect to $\mathcal{K}$ if $F$ is G-semismooth at $x\in \mathcal{D}$ and for any $J \in \mathcal{K}(x+d)$,
    \[
    F(x+d) - F(x) - J d = \mathcal{O}(\|d\|^2) \ \mathrm{as}\ d\to 0.
    \]
If $F$ is G-(strongly) semismooth at any $x \in \mathcal{D}$ with respect to $\mathcal{K}$, then $F$ is called a G-(strongly) semismooth function with respect to $\mathcal{K}$.
\end{definition}

Next, we give the semismooth implicit function theorem that it is used later to prove the local quadratic convergence result. It is a rearrangement of~\cite[Theorem~4]{Gowda2004} and~\cite[Theorem 6]{Pang2003}. The proofs are given in Appendix~\ref{appA}. 

\begin{corollary}[Semismooth Implicit Function Theorem]\label{coro-2-1}
 Suppose that $F: \mathbb{R}^n\times \mathbb{R}^m \to \mathbb{R}^m$ is a (strongly) semismooth function with respect to $\partial_{\mathrm{B}} F$ in an open neighborhood of $(x^0,y^0)$ with $F(x^0,y^0) = 0$. Let $H(y) = F(x^0, y)$. If every matrix in $\partial H(y^0)$ is nonsingular, then there exists an open set $\mathcal{V} \subset \mathbb{R}^n$ containing $x^0$, a set-valued function $\mathcal{K}:\mathcal{V} \rightrightarrows \mathbb{R}^{m \times n}$, and a G-(strongly) semismooth function $f:\mathcal{V} \to \mathbb{R}^m$ with respect to $\mathcal{K}$ 
satisfying $f(x^0) = y^0$, 
\[
F\left(x, f(x)\right) = 0
\]
for every $x\in \mathcal{V}$ and the set-valued function $\mathcal{K}$ is
\[
\mathcal{K}:x \mapsto \left\{-(A_y)^{-1} A_x: [A_x\ A_y] \in \partial_{\mathrm{B}} F \big(x,f(x)\big)\right\},
\]
where the map $x\mapsto \mathcal{K}(x)$ is compact valued and upper semicontinuous. 
\end{corollary}

Note that the set-valued function $\mathcal{K}$ in Corollary~\ref{coro-2-1} is not necessarily $\partial_{\mathrm{B}} f(x)$ nor $\partial f(x)$.

\section{A Riemannian Proximal Newton Method}\label{sec:RPN}

In this section, the proposed Riemannian proximal Newton method is described in detail and the local quadratic convergence rate is established.

\subsection{Algorithm Interpretation} \label{sec:RPNdescription}

We propose a Riemannian proximal Newton (RPN) method stated in Algorithm~\ref{alg:RPN}. Each iteration of RPN consists of three phases: the computation of a Riemannian proximal gradient direction, the Riemannian Newton modification of the step direction, and finally, the retraction on the manifold constraint.

\begin{algorithm}
\caption{A Riemannian proximal Newton method (RPN)}
\label{alg:RPN}
\begin{algorithmic}[1] 
\Require A $(n-d)$-dimensional embedded submanifold of $\mathbb{R}^n$, $x_0 \in \mathcal{M}$, $t > 0$;
\For{$k = 0,1,\dots$}
\State \label{alg:RPN:st01} Compute $v(x_k)$ by solving 
\begin{equation} \label{eq:subforv}
v(x_k) = \argmin_{v \in \T_{x_k} \M}\ f(x_k) + \nabla f(x_k)^{\T} v + \frac{1}{2t} \|v\|_{\F}^2 + h(x_k + v).
\end{equation}
\State \label{alg:RPN:st02} Find $u(x_k) \in \T_{x_k}\M$ by solving
\begin{equation}\label{3-3}
    J(x_k) [u(x_k)] = - v(x_k),
\end{equation}
where 
\begin{equation}\label{eq:Jsol}
J(x_k) = -\left[\I_n - \Lambda_{x_k} + t\Lambda_{x_k} (\nabla^2 f(x_k) - \mathcal{L}_{x_k})\right],
\end{equation}
$\Lambda_{x_k} = M_{x_k} -  M_{x_k} B_{x_k} H_{x_k} B_{x_k}^{\T} M_{x_k}$, $H_{x_k} = \left(B_{x_k}^{\T} M_{x_k} B_{x_k}\right)^{-1}$, $B_{x_k}$ is an orthonormal basis of $\N_{x_k} \mathcal{M}$, $\mathcal{L}_{x_k}(\cdot) = \W_{x_k}\big(\cdot,B_{x_k} \lambda(x_k)\big)$, $\W_{x_k}$ denotes the Weingarten map~\eqref{2-3}, $\lambda(x_k)$ is the Lagrange multiplier~\eqref{3-1} at $x_k$, and $M_{x_k}$ is a diagonal matrix defined in~\eqref{eq:SpecialM}.
\State $x_{k+1} = R_{x_k}\left( u(x_k)\right)$;
\EndFor
\end{algorithmic}
\end{algorithm}

Firstly, in Step~\ref{alg:RPN:st01}, we compute the Riemannian proximal gradient direction $v(x_k)$ as done in~\cite{2020Proximal}. It has been shown that $v(x_k)$ can be efficiently computed by applying a semismooth Newton method to the KKT conditions of the optimization problem in~\eqref{eq:subforv}. Specifically, the KKT condition for Problem~\eqref{eq:subforv} is given by
\begin{equation}\label{eq:KKT01}
\partial_v \mathscr{L}_k(v, \lambda) = 0, \hbox{ and } B_{x_k}^T v = 0, 
\end{equation}
where we use the fact that $v \in \T_{x_k} \mathcal{M}$ is equivalent to $B_{x_k}^T v = 0$, $\mathscr{L}_k$ is the Lagrange function given by $\mathscr{L}_k(v, \lambda) = f(x_k) + \nabla f(x_k)^{\T} v + \frac{1}{2t} \|v\|_{\F}^2 + h(x_k + v) + \lambda^T B_{x_k}^{\T} v$, and $\lambda \in \mathbb{R}^d$ denotes the Lagrange multiplier. It follows from~\eqref{eq:KKT01} that
\begin{equation}\label{3-1}
v=\prox_{t h} \bigl( x_k - t \left[\nabla f(x_k) + B_{x_k} \lambda \right] \bigr) - x_k, \hbox{ and } B_{x_k}^{\T} v = 0,
\end{equation}
where $ \mathrm{prox}_{t h}(z) $ denotes the proximal mapping of $t h$, i.e., 
\begin{equation} \label{eq:EProx}
\mathrm{prox}_{t h}(z) = \argmin_{x \in \mathbb{R}^n} \frac{1}{2} \|x - z\|^2 + t h(x) = \max(|z| - t \mu, 0) \odot \sign(z).
\end{equation}
From~\eqref{3-1}, we have an equation of $\lambda$ given by
\begin{equation} \label{eq:Psilambda}
B_{x_k}^T \left( \prox_{t h} \bigl( x_k - t \left[\nabla f(x_k) + B_{x_k} \lambda \right] \bigr) - x_k \right) = 0,
\end{equation}
which can be solved efficiently by a semismooth Newton method and the resulting $v(x_k)$ is obtained by the first equation of~\eqref{3-1}.

Step~\ref{alg:RPN:st02} is the main innovation of our proposed RPN compared to the first-order method of~\cite{2020Proximal} and it is the Riemannian analogue of the semismooth Newton method applied to $f(x) + \mu \|x\|_1$ in the Euclidean setting~\cite{Xiao2018}. For smooth optimization problems, the direction $v(x_k)$ plays the same role as the negative Riemannian gradient. To see that, if the nonsmooth term $h$ is the zero function, i.e., $h(x) \equiv 0$, and $t$ is chosen to be one, then $v(x_k)$ is the negative Riemannian gradient, i.e., $v(x_k) = - \grad f(x_k)$. The Newton direction is given by solving the Newton equation
\[
\Hess f(x_k) [\eta_k] = - \grad f(x_k),
\]
which can be written in terms of $v(x_k)$, i.e.,
\begin{equation}\label{smooth_sear}
\Proj_{x_k} (\mathrm{D} v(x_k) [\eta_k]) = - v(x_k),
\end{equation}
since $v(x_k) = - \grad f(x_k)$. The same idea can be used for $v(x_k)$ given by~\eqref{eq:subforv}. However, $v(x_k)$ is not a smooth function of $x_k$. In Subsections~\ref{sec:RPNsemiv}, \ref{sec:RPNgenJv} and \ref{sec:RPNlocal}, we will show that given a local optimal point $x_*$, (i) $v(x)$ is a G-strongly semismooth function of $x$ in a neighborhood of $x_*$, (ii) the linear operator $J(x):\T_{x}\M \to \T_x \M$ in~\eqref{eq:Jsol} is motivated by a generalized Jacobian of $v(x)$, and (iii) the direction $u(x)$ given by~\eqref{3-3} is a quadratic convergence direction.

\subsection{G-strongly semismoothness of $v(x)$} \label{sec:RPNsemiv}

The first step of the analysis is to prove the G-strongly semismoothness of the term $v(x_k)$, which herein we denote as $v(x)$ for a concise notation. The G-strongly semismoothness of $v(x)$ is proven by verifying the assumptions of the semismooth implicit function theorem in Corollary~\ref{coro-2-1} at $x_*$. Define the function
\[
\mathcal{F}: \mathbb{R}^n \times \mathbb{R}^{n + d} \mapsto \mathbb{R}^{n + d}: (x; v, \lambda) \mapsto 
\begin{pmatrix}
 v + x - \prox_{t h} \bigl( x - t [\nabla f(x) + B_x \lambda ] \bigr)\\
B_x^{T} v 
\end{pmatrix}.
\]
Since $\mathcal{F}$ is piecewise smooth, it follows from~\cite[Proposition 2.26]{ulbrich2011} 
that
\begin{equation} \label{eq:semicalF}
\mathcal{F} \hbox{ is a strongly semismooth function with respect to } \partial_{\mathrm{B}} \mathcal{F}.  
\end{equation}
Let $x_*$ be a local optimal point. It has been shown in~\cite[Lemma~5.2]{2020Proximal} that the solution $v(x)$ of~\eqref{eq:subforv} is a strict decent direction. Hence 
\begin{equation}\label{v_*}
    v(x_*) = 0.
\end{equation}

It follows from the KKT condition~\eqref{3-1} that
\begin{equation} \label{eq:calFeq0}
\mathcal{F}(x_*; v(x_*), \lambda(x_*)) = 0,
\end{equation}
where $\lambda(x_*)$ also is the solution of~\eqref{eq:Psilambda} at $x_*$. We use $v_*$ and $\lambda_*$ to respectively denote $v(x_*)$ and $\lambda(x_*)$ for simplicity. It follows that $\mathcal{F}(x_*; v_*, \lambda_*) = 0$. 

Define
\[
\begin{aligned}
\mathcal{G}&:\mathbb{R}^{n+d}\to \mathbb{R}^{n+d} \\
&: (v,\lambda) \mapsto \mathcal{G}(v, \lambda) = \mathcal{F}(x_*;v,\lambda) = \begin{pmatrix}
 v + x_* - \prox_{t h} \bigl( x_* - t[\nabla f(x_*) + B_{x_*} \lambda]\bigr)\\
B_{x_*}^{T} v 
\end{pmatrix}.
\end{aligned}
\]
Next, we will show that under a certain reasonable assumption, any entry in $\partial \mathcal{G}(v_*,\lambda_*)$ is nonsingular. Since the proximal mapping $\prox_{t h}$ is given by~\eqref{eq:EProx}, it follows from the chain rule~\cite[Theorem~2.3.9]{clarke1990} that the Clarke subdifferential of $\mathcal{G}$ is given by
\begin{equation} \label{eq:partialG}
\partial \mathcal{G}(v_*,\lambda_*) = \left\{ \begin{bmatrix} I_n
 & t M B_{x_*} \\
B_{x_*}^{\T} & 0_d  
\end{bmatrix}: M \in \partial  \prox_{t h} \bigl( x_* - t [\nabla f(x_*)  + B_{x_*} \lambda_*]\bigr)\right\},
\end{equation}
where 
\[
\begin{aligned}
\partial \prox_{t h} (z) = \{ M \in \mathbb{R}^{n \times n} \hbox{ is diagonal}: &M_{ii} = 1 \hbox{ if } |z_i| > t \mu, M_{ii} = 0 \hbox{ if } |z_i| < t \mu,  \\ 
&\left. \hbox{ and } M_{i i} \in [0, 1] \hbox{ otherwise.} \right\}.
\end{aligned}
\]
We consider a representative matrix in $\partial \prox_{t h}(x - t [\nabla f(x)  + B_{x} \lambda(x)])$, denoted by $M_{x}$, i.e.,
\begin{equation} \label{eq:SpecialM}
(M_{x})_{i i} = 
\left\{
\begin{array}{cc}
    0 & \hbox{ if $|x - t [\nabla f(x)  + B_{x} \lambda(x)]|_i \leq t \mu$; } \\
    1 & \hbox{ otherwise.}
\end{array}
\right.
\end{equation}
We let $j$ denote the number of nonzero entries in $M_{x_*}$ and, without loss of generality, we assume that the nonzero entries of $M_{x_*}$ lie on the left upper corner of the matrix, i.e.,
\begin{equation} \label{eq:03}
M_{x_*} = \begin{bmatrix}
 \I_j & \\
 & 0_{n-j} 
\end{bmatrix}.
\end{equation}
The nonsingularity of all entries in $\partial \mathcal{G}(v_*, \lambda_*)$ relies on Assumption~\ref{assu-3-1}. In addition, we provide an alternative interpretation for Assumption~\ref{assu-3-1}, see Remark~\ref{con_3_21}.

\begin{assumption}\label{assu-3-1}
 Let $B_{x_*}^{\T} = [\bar{B}_{x_*}^{\T}, \hat{B}_{x_*}^{\T}] $, where $\bar{B}_{x_*} \in \mathbb{R}^{j \times d}$ and $\hat{B}_{x_*} \in \mathbb{R}^{(n-j)\times d}$. It is assumed that $j \geq d$ and $\bar{B}_{x_*}$ is full column rank.
\end{assumption}

 In view of Lemma~\ref{lemm_Mx}, it follows from $v(x_*)=0$ of~\eqref{v_*} that if  the matrix $M_{x_*}$ is in the form of~\eqref{eq:03}, then $x_*$ can be written in the form of  $x_* = [\bar{x}_*^T, 0^T]^T$, where any entries in $\bar{x}_* \in \R^{j}$ are nonzero. Therefore,
if the manifold $\M$ is the unit sphere $\mathbb{S}^{n-1}$, then $d = 1$ and $B_{x_*} = x_*$ with $\bar{B}_{x_*} = \bar{x}_*$. Since $x_* \in \mathbb{S}^{n-1}$ with $x_* \ne 0$, then $j \geq 1$ and $\bar{B}_{x_*}$ is full column rank, i.e., Assumption~\ref{assu-3-1} holds for the unit sphere $\mathbb{S}^{n-1}$. It also can be shown that Assumption~\ref{assu-3-1}holds for the oblique manifold $\mathrm{O B}(p,n) = (\mathbb{S}^{n-1})^p$.

\begin{lemma}\label{lemm_Mx}
Suppose that $v$ is the solution of~\eqref{eq:subforv} at $x$, then $(x + v)_i = 0$ if and only if $(M_x)_{i i} = 0$, and $(x + v)_i \neq 0$ if and only if $(M_x)_{i i} = 1$, where $M_x$ is defined in~\eqref{eq:SpecialM}. 
\end{lemma}
\begin{proof}
According to the first equality in~\eqref{3-1}, we have
\begin{equation}
x + v =\prox_{th} \bigl( x  - t[\nabla f(x) + B_x \lambda] \bigr),\label{3-9}
\end{equation}
where the $\prox_{t h}(z)$ is given by \eqref{eq:EProx} with $z = x  - t[\nabla f(x) + B_x \lambda ]$. 
Thus,
\begin{itemize}
    \item[(1)] if $(M_x)_{i,i} = 1$, then $|z_i|> t \mu$. It follows from~\eqref{3-9} that 
    $
    x_i + v_i = z_i - t\mu \mathrm{sgn}(z_i) \neq 0;
    $
    \item[(2)] if $(M_x)_{i,i} = 0$, then $|z_i|\le t\mu$. It follows from~\eqref{3-9} that 
    $
    x_i + v_i = 0.
    $
\end{itemize}
Therefore, $(x + v)_i = 0$ if and only if $(M_x)_{i i} = 0$, and $(x + v)_i \neq 0$ if and only if $(M_x)_{i i} = 1$.
\end{proof}

\begin{lemma}\label{lemm-3-1}
If Assumption~\ref{assu-3-1} holds, then every matrix in $\partial \mathcal{G}(v_*,\lambda_*)$ is invertible.
\end{lemma}
\begin{proof}

For any $M \in \partial  \prox_{t h} \bigl( x_* - t [\nabla f(x_*)  + B_{x_*} \lambda_*]\bigr)$, 
without loss of generality, assume that $M = \diag(s)$, where $s^{\T} = [1,\cdots,1,s_1,\cdots,s_{\ell},0,\cdots,0]$, where $s_i \in (0,1)$, $i=1\dots,\ell$, the number of element 1 is $j$. According to the partition of $M$, let $B_{x_*}^{\T} = [\bar{B}_{x_*}^{\T},\tilde{\bar{B}}_{x_*}^{\T}, \tilde{\hat{B}}_{x_*}^{\T}   ]$, where $\bar{B}_{x_*} \in \R^{j\times d}$,  $\tilde{\bar{B}}_{x_*} \in \mathbb{R}^{\ell\times d}$ and $\tilde{\hat{B}}_{x_*} \in \mathbb{R}^{(n-j-\ell)\times d}$.  Note that $B_{x_*}^T M B_{x_*} = \begin{bmatrix}
  \bar{B}_{x_*} \\  S\tilde{\bar{B}}_{x_*}
\end{bmatrix}^{\T} \begin{bmatrix}
  \bar{B}_{x_*} \\  S\tilde{\bar{B}}_{x_*}
\end{bmatrix}$, where $S = \diag(\sqrt{s_1},\cdots,\sqrt{s_{\ell}})$. Since $\bar{B}_{x_*}$ is full column rank by Assumption~\ref{assu-3-1}, then $\begin{bmatrix}
  \bar{B}_{x_*} \\  S\tilde{\bar{B}}_{x_*}
\end{bmatrix}$ is also full column rank, then 
$B_{x_*}^T M B_{x_*}$ is invertible.

For any $D \in \partial \mathcal{G}(v_*,\lambda_*)$, it follows from~\eqref{eq:partialG} that
\[
D = \begin{bmatrix} I_n
 & t M B_{x_*} \\
B_{x_*}^{\T} & 0_d  
\end{bmatrix}.
\]
One can verify that the matrix
\[
\begin{bmatrix}
\I_n - M B_{x_*} H_{x_*} B_{x_*}^{\T} & M B_{x_*} H_{x_*}\\
\frac1t H_{x_*} B_{x_*}^{\T}  & -\frac1t H_{x_*} 
\end{bmatrix},
\]
is the inverse of $D$, where $H_{x_*} = \left(B_{x_*}^{\T} M B_{x_*}\right)^{-1}$. Thus, every matrix in $\partial \mathcal{G}(v^*,\lambda_*)$ is invertible.
\end{proof}

It follows from~\eqref{eq:semicalF}, \eqref{eq:calFeq0}, and Lemma~\ref{lemm-3-1} that the assumptions of Corollary~\ref{coro-2-1} hold. As a consequence of Corollary~\ref{coro-2-1}, we have that there exists a neighborhood $\U$ of $x_*$ such that there exists a G-strongly semismooth function $S:\mathcal{U} \rightarrow \mathbb{R}^{n + d}:x \mapsto S(x) = (v(x), \lambda(x))$ with respect to $\mathcal{K}_S$ such that for every $x \in \mathcal{U}$,
\[
\mathcal{F}\left(x; S(x)\right) = 0,
\]
and the set-valued function $\mathcal{K}_S$ is
\[
\mathcal{K}_S: x \mapsto \left\{ - B^{-1} A : [A \; B] \in \partial_{\mathrm{B}} \mathcal{F}(x; S(x)) \right\}.
\]
Thus, $v:\mathcal{U} \rightarrow \mathbb{R}^n:x \mapsto v(x)$ is a G-strongly semismooth function with respect to $\K_v$, where
\begin{equation}\label{genJac_v}
\K_v:x \mapsto \left\{-[\I_n,\ 0]\cdot B^{-1} A: [A\ B] \in\partial_{\mathrm{B}} \mathcal{F} \big(x;S(x)\big)\right\}.
\end{equation}
Given $x \in \mathcal{U}$, any element of $\mathcal{K}_v(x)$ is called a generalized Jacobian of $v$ at $x$.

\subsection{The linear operator $J(x)$}  \label{sec:RPNgenJv} 

In this section, we derive a semismooth analogue of the Riemannian Newton direction. Recall that the classical smooth Riemannian Newton direction satisfies~\eqref{smooth_sear}, i.e., 
\[
\Proj_{x} (\mathrm{D} v(x) [\eta(x)]) = - v(x),
\]
where $\mathrm{D} v(x) [\eta(x)] = (J_x v)\eta_x$ involves the Jacobian of $v$ at $x$. For the nonsmooth case, the main task is to design a linear operator
$J(x): \T_x \M \rightarrow \T_x \M$, where $J(x)$ is related to the generalized Jacobian of $v$ at $x$.
To this end, we need to first compute a generalized Jacobian of $v$. For $x\in \mathcal{U}$, we select a matrix $[A\ B] \in \partial_{\mathrm{B}} \mathcal{F} \big(x,S(x)\big)$, that is 
\begin{equation*}
    A = \begin{bmatrix}
    I_n - M_{x} \big(I_n - t\nabla^2 f(x)\big) + t M_x (\mathrm{D} B_x) \lambda\\
    (\mathrm{D} B_x^{\T}) v  \end{bmatrix},\quad
    B =  \begin{bmatrix} I_n
    & t M_{x} B_{x} \\
    B_{x}^{\T} & 0_d  
    \end{bmatrix},
\end{equation*}
where $M_x$ is given in~\eqref{eq:SpecialM}. Therefore the generalized Jacobian has the following form
\begin{equation}
\begin{aligned}
\mathcal{J}_x &\overset{\eqref{genJac_v}}{=} - [\I_n,\ 0]\cdot B^{-1} A \\
&= -\Bigl[\I_n - M_x B_x H_x B_x^{\T}  - \Lambda_x (\I_n - t \nabla^2 f(x)) + t \Lambda_x (\mathrm{D}B_x) \lambda + M_x B_x H_x (\mathrm{D} B_x^{\T}) v \Big]\\
&= -\left[ \I_n - \Lambda_x + t\Lambda_x \left(\nabla^2 f(x) + (\mathrm{D}B_x) \lambda \right) \right]  -\left[M_x B_x H_x (\mathrm{D} B_x^{\T} )v - M_x B_x H_x B_x^{\T}\right],
\end{aligned}\label{3-6}
\end{equation}
where $\Lambda_x = M_x -  M_x B_x H_x B_x^{\T} M_x$ and $H_x = \left(B_{x}^{\T} M_{x} B_{x}\right)^{-1}$. 
Thus, $\mathcal{J}_x \in \mathcal{K}_v (x)$ is a generalized Jacobian of $v$ at $x$. For any $\omega \in \T_x \M$, the action of $\mathcal{J}_x$ is given by
\begin{equation}\label{eq:actionJx}
\mathcal{J}_x [\omega] = -\omega + \Lambda_x (\I_n - t \nabla^2 f(x))\omega - t \Lambda_x (\mathrm{D}B_x[\omega]) \lambda - M_x B_x H_x ( \mathrm{D}B_x^{\T}[\omega]) v.
\end{equation}
The differential of $B_{x}$ requires the information of an orthonormal basis of the normal space $\N_x \M$ at $x$, which may not be readily available. Lemma \ref{lemm-3-2} shows that the term including the differential of $B_{x}$ can be written in term of the Weingarten map.

\begin{lemma}\label{lemm-3-2}
Suppose that $B_x$ is an orthonormal basis of $\N_x \M$ with $\dim \N_x \M = d$. Then 
\[
\Lambda_x (\mathrm{D} B_x[\omega])\lambda = -\Lambda_x \mathcal{W}_x(\omega,B_x \lambda),
\]
where $\omega\in \T_x \M$, $\lambda \in \R^d$ and $\W_x$ denotes the Weingarten map defined in~\eqref{2-3}.
\end{lemma}

\begin{proof}
Since $\Proj_{x}^{\perp}  = B_x B_x^{\T}$ and $\Lambda_x B_x = 0$, we have
\[
\Lambda_x (\mathrm{D}B_x[\omega]) \lambda = \Lambda_x \Proj_{x}\left( \mathrm{D}B_x[\omega] \lambda\right).
\]
By the definition of  \textit{Weingarten map} in~\eqref{2-3}, we have $\mathcal{W}_x(\omega,u) = \Proj_x \big(\mathcal{W}_x(\omega,u)\big)$, where $\omega \in \T_x \M$, $u\in \N_x \M$. Therefore, it holds that
\begin{equation}\label{eq:01}
\begin{aligned}
  \mathcal{W}_x(\omega,u) &= \Proj_x\big(\D(x\mapsto \Proj_x)(x)[\omega] \cdot u\big) = \Proj_x\big(\D(I - B_x B_x^T)[\omega] \cdot u\big) \\
  &= \Proj_x\big(- \mathrm{D}B_x[\omega]\cdot B_x^{\T}u - B_x\cdot \mathrm{D}B_x^{\T}[\omega] \cdot u \big ) \\
  &= - \Proj_x\big( \mathrm{D}B_x[\omega]\cdot B_x^{\T}u \big), 
\end{aligned}
\end{equation}
where the sign $\cdot$ denotes usual matrix multiplication,  the last equality follows from $B_x\cdot \mathrm{D}B_x^{\T}[\omega] \cdot u \in \mathrm{N}_{x} \mathcal{M}$. Letting $u = B_x \lambda \in \mathrm{N}_{x} \mathcal{M}$ in~\eqref{eq:01} yields
$\mathcal{W}_x(\omega,B_x \lambda) = - \Proj_x\big( \mathrm{D}B_x[\omega]\cdot \lambda \big)$, which implies
$\Lambda_x ( \mathrm{D}B_x[\omega]) \lambda = -\Lambda_x \mathcal{W}_x(\omega,B_x \lambda).$
\end{proof}

By Lemma~\ref{lemm-3-2}, the action of $\mathcal{J}_x$ can be reformulated as
\begin{equation} \label{eq:02}
\mathcal{J}_x [\omega] 
= -\left[ \I_n - \Lambda_x + t\Lambda_x (\nabla^2 f(x) - \mathcal{L}_x)\right]\omega - M_x B_x H_x ( \mathrm{D}B_x^{\T}[\omega]) v,
\end{equation}
where $\mathcal{L}_x(\omega) = \W_x(\omega,B_x \lambda)$ is linear operator with respect to $\omega\in \T_x \M$. Since at a stationary point $x_*$, $v(x_*)$ is equal to zero by~\cite{2020Proximal}, the last term in~\eqref{eq:02} can be dropped without influencing the local quadratic convergence rate. This yields the linear operator $J(x)$ used in Algorithm~\ref{alg:RPN}. Since for any $\omega \in \T_x \M$, we have $B_x^{\T}J(x)[w] = 0$, it follows that $J(x): \T_x\M \to \T_x\M$.
\begin{remark} \label{example}
If the manifold $\M$ is the unit sphere $\mathbb{S}^{n-1}$, then by~\cite{absil2013extrinsic}, we have $\mathcal{W}_x(w, u) = -wx^{\T}u$ for any $w\in \T_x \M$, $u\in \N_x \M$. For $x\in \mathbb{S}^{n-1}$, we have $B_x = x$. Without  loss of generality, we assume that 
$M_x = \begin{bmatrix}
    \I_j &\\ & 0_{n-j} \end{bmatrix}$.
According to the partition of $M_x$, we have 
\[
\begin{aligned}
    x = \begin{bmatrix}
    x_j \\ x_{n-j}
    \end{bmatrix}, &\ \
    \nabla^2 f(x) = \begin{bmatrix}
    H_{x}^{(11)} & H_x^{(12)}\\
    H_x^{(21)} & H_x^{(22)} 
    \end{bmatrix}, & \mathcal{L}_{x}(\cdot) = \W_{x}\big(\cdot,B_{x} \lambda(x)\big) = -\lambda(x) \I_n,
\end{aligned}
\]
where $H_{11} \in \mathbb{R}^{j\times j}$. By~\eqref{eq:Jsol}, we have
\[
\begin{aligned}
    J(x) &= -\left[\I_n - \Lambda_{x} + t\Lambda_{x} \left(\nabla^2 f(x) - \mathcal{L}_{x}\right)\right]\\
    &= -\begin{bmatrix}
   \frac{x_j x_j^{\T}}{x_j^{\T} x_j} + t\left(\I_j - \frac{x_j x_j^{\T}}{x_j^{\T} x_j}\right)\Bigl(H_x^{(11)}+\lambda(x) \I_j\Bigr) & t\left(\I_j -\frac{x_j x_j^{\T}}{x_j^{\T} x_j}\right)H_x^{(12)}\\
   0_{(n-j) \times j}& \I_{n-j}
   \end{bmatrix}.
\end{aligned}
\]
\end{remark}

\begin{remark}\label{smoothcase}
Consider the smooth case in~\eqref{1-1}, i.e., $h(x) \equiv 0$. The KKT condition in~\eqref{3-1} is therefore
$\nabla f(x) + \frac1t v  + B_x \lambda = 0,\  B_x^{T} v = 0$.
The closed-form solution is given by 
$\lambda = - B_x^{\T} \nabla f(x), \ v = - t \grad f(x)$. 
For the smooth case, we have $M_x = \I_n$, then $J(x)$ in~\eqref{eq:Jsol} can be simplified to  $J(x) = B_x B_x^{\T} -t(\I_n - B_x B_x^{\T})(\nabla^2 f(x) - \mathcal{L}_x)$. Therefore, for any $\omega \in \T_x \M$,
\[
J(x)[\omega] = -t(\I_n - B_x B_x^{\T})(\nabla^2 f(x) - \mathcal{L}_x)\omega = -t \Hess f(x)[\omega],
\]
where the last equation follows from~\cite{absil2013extrinsic}.
It follows that the linear system $J(x)[u(x)] = -v(x)$ is equivalent to the Riemannian Newton linear system 
\[
\Hess f(x)[u(x)] = -\grad f(x).
\]
Thus, $u(x)$ is the Riemannian Newton direction.
\end{remark} 

\subsection{Local Convergence Analysis} \label{sec:RPNlocal}

In this section, we show that for a sufficiently small neighborhood of a local optimal point $x_*$, the RPN has quadratic convergence. Without loss of generality, we assume that the nonzero entries of $x_*$ are in the first part, i.e., $x_* = [\bar{x}_*^T, 0^T]^T$, where any entries in $\bar{x}_* \in \R^{j}$ are nonzero. Moreover, it follow from Lemma~\ref{lemm_Mx} together with $v(x_*)=0$ of~\eqref{v_*} {that if $x_*$ is written in the form of $[\bar{x}_*^T, 0^T]^T$, then the matrix $M_{x_*}$ is in the form of}~\eqref{eq:03}.

 Our analysis is based on an assumption that in a sufficiently small neighborhood, RPN will produce iterates with the correct support of the stationary point $x_*$, which will achieve quadratic convergence.

\begin{assumption}\label{assu-3-2}
There exists a neighborhood $\mathcal{U}$ of $x_* = [\bar{x}_*^T, 0^T]^T$ on $\M$ such that for any $x = [\bar{x}^T, \hat{x}^T]^T \in \mathcal{U}$, it holds that $\bar{x} + \bar{v} \neq 0$ and $\hat{x} + \hat{v} = 0$, where $v = [
\bar{v}^T, \hat{v}^T ]^T$ denotes the solution of~\eqref{eq:subforv} at $x$.
\end{assumption}

Since $x + v$ is promoted to be sparse by the term $h(x) = \mu \|x\|_1$ in Problem~\eqref{eq:subforv}, and other terms therein are smooth, it is reasonable to assume that the support does not change in a neighborhood of $x_*$.

Lemma \ref{lemm-3-3} shows that when the support of $x + v$ is the one of $x_*$, i.e., the support does not change, the locations of the nonzero entries of $M_x$ also remain the same. 

\begin{lemma}\label{lemm-3-3}
Under Assumption~\ref{assu-3-2},  for any $x\in \mathcal{U}$, we have $M_x = M_{x_*}$, where $M_x$ is defined in~\eqref{eq:SpecialM}.
\end{lemma}
\begin{proof}
According to Lemma~\ref{lemm_Mx}, we have 
$(x + v)_i = 0$ if and only if $(M_x)_{i i} = 0$, and $(x + v)_i \neq 0$ if and only if $(M_x)_{i i} = 1$. It follows from Assumption~\ref{assu-3-2} that $M_x = M_{x_*}$ for any $x\in \U$.
\end{proof}

Now, we are ready to give the local convergence analysis of Algorithm~\ref{alg:RPN}, with the assumption that $J(x_*)$ is nonsingular, where $x_*$ is a local optimal point of~\eqref{1-1}. 
We claim that, under the condition of Proposition~\ref{sec_opt} given later, $J(x_*)$ is always nonsingular. 

\begin{theorem}\label{the_super}
    Suppose that $x_*$ is a local optimal point of \eqref{1-1}, that Assumption~\ref{assu-3-1} and Assumption~\ref{assu-3-2} hold, and that $J(x_*)$ is nonsingular. Then there exists a neighborhood $\mathcal{V}$ of $x_*$ on $\M$ such that for any $x_0 \in \mathcal{V}$, Algorithm~\ref{alg:RPN} generates a sequence $\{x_k\}$ converging quadratically to $x_*$. 
\end{theorem}

\begin{proof}
Since $v$ is G-strongly semismooth with respect to $\K_v$, it holds that
\[
v(x) - v(x_*) - \mathcal{J}_x(x - x_*) = \mathcal{O}(\|x - x_*\|^2),
\]
where $\mathcal{J}_x = \J_1(x) + \J_2(x)\in \K_v(x)$ in~\eqref{3-6} with
\[
\J_1(x) = -\left[\I_n - \Lambda_x + t\Lambda_x \left(\nabla^2 f(x) + (\mathrm{D} B_x) \lambda \right)\right]
\]
and
\[
\J_2(x) = -\left[M_x B_x H_x (\mathrm{D} B_x^{\T} )v - M_x B_x H_x B_x^{\T}\right].
\]  
According to Lemma~\ref{lemm-3-3}, for $x\in \U$, the position of $M_x$ is fixed, it follows that $\J_1(x)$ and $\J_2(x)$ vary continuously with $x$. Since $v(x)$ is G-strongly semismooth, it follows from \cite[Corollary 1]{Gowda2004} that there exists $\gamma > 0$ such that $\|v(x) - v(x_*) \| \le \gamma \|x-x_*\|$. Therefore,
\[
\begin{aligned}
    \J_2 (x - x_*) &=  -M_x B_x H_x\cdot \mathrm{D}B_x^{\T}[x-x_*] \cdot v + M_x B_x H_x B_x^{\T}(x-x_*) \\
    &\overset{(i)}{=}  -M_x B_x H_x\cdot \mathrm{D}B_x^{\T}[x-x_*] \cdot v - M_x B_x H_x B_x^{\T}\left(R_x^{-1}(x_*) + \mathcal{O}(\|x-x_*\|^2\right)\\
    &\overset{(ii)}{=} -M_x B_x H_x\cdot \mathrm{D}B_x^{\T}[x-x_*] \cdot v + \mathcal{O}(\|x-x_*\|^2)\\
    &\overset{(iii)}{=}\mathcal{O}(\|x-x_*\|^2),
\end{aligned}
\]
where $(i)$ follows from $x_* - x = R_x^{-1}(x_*) + \mathcal{O}(\|x-x_*\|^2)$ and $R_x^{-1}(x_*) \in \T_x \M$ is inverse retraction,$ (ii)$ and $(iii)$ follow from $M_x$ being fixed, $R_x^{-1}(x_*)$  belonging to $\T_x \M$ and therefore $B_x^{\T}\left(R_x^{-1}(x_*)\right) = 0$, $B_x$ and $H_x$ varying smoothly when $x$ is sufficiently close to $x_*$. Therefore, 
\[
v(x) - v(x_*) - \J_1(x)(x-x_*) = \mathcal{O}(\|x - x_*\|^2).
\]
Note that
\[
\begin{aligned}
&\J_1(x)(x-x_*) = -\J_1(x)\left[R_x^{-1}(x_*) + \mathcal{O}(\|x-x_*\|^2\right]\\
=& -J(x)\left[R_x^{-1}(x_*)\right] + \mathcal{O}\left(\|x-x_*\|^2\right)= J(x)(x-x_*) + \mathcal{O}\left(\|x-x_*\|^2\right),
\end{aligned}
\]
then 
\[
v(x) - v(x_*) - J(x)(x-x_*) = \mathcal{O}(\|x - x_*\|^2).
\]
Since $J(x)u(x)=-v(x)$, then
\[
-J(x) \left[u(x) + x - x_*\right] = \mathcal{O}(\|x-x_*\|^2).
\]
Since $J(x)$ in~\eqref{eq:Jsol} varies continuously with $x$ and $J(x_*)$ is nonsingular, $J(x)$ and $J(x)^{-1}$ are bounded.
Thus,  
$x+u(x) - x_* = \mathcal{O}(\|x-x_*\|^2).$
By adding the subscript, we have $x_k+u_k - x_* = \mathcal{O}(\|x_k-x_*\|^2)$. Therefore, for any $\eta >0$, there exists a neighborhood $\U_1$ of $x_*$ such that for any $x_k \in \U_1$, it holds that
$\|x_k + u_k  -x_*\| \le \eta  \|x_k - x_*\|^2 \le \eta  \|x_k - x_*\|$. Therefore, we have $\|u_k\| \le (1+\eta)\|x_k - x_*\|$,
which means $u_k = \mathcal{O}(\|x_k-x_*\|)$. According to the definition of the retraction, there exists a neighborhood $\U_2$ of $x_*$ such that for any $x_k \in \U_2$,
\[
R_{x_k}(u_k) - (x_k + u_k) = \mathcal{O}(\|u_k\|^2).
\]
 Let $\mathcal{V} = \U \cap \U_1 \cap \U_2$, for any $x_k \in \mathcal{V}$, we have 
\[
R_{x_k}(u_k) - (x_k + u_k) = \mathcal{O}(\|x_k - x_*\|^2),
\]
and
\[
\begin{aligned}
    \|x_{k+1} - x_*\| &= \|R_{x_k}(u_k) - x_*\|\\
    &\le \|R_{x_k}(u_k) - (x_k + u_k)\| + \|x_k + u_k -x_*\| = \mathcal{O}(\|x_k - x_*\|^2).
\end{aligned}
\]
\end{proof}

\begin{remark}\label{im_thm}
    In Theorem~\ref{the_super}, we achieve local quadratic convergence results within the neighborhood $\mathcal{V}$, which plays a key role in the design of globalization methods. The main idea behind globalization is to enable $x_k$ to enter the neighborhood $\mathcal{V}$, we then propose a globalized version, see Algorithm~\ref{alg:RPN_s}.
\end{remark}

\subsection{Optimality Conditions}\label{subsec:opc}

The first-order necessary optimality condition for Stiefel manifold has been given in~\eqref{v_*}, which shows that  $v(x_*) = 0$ if $x_*$ is a local optimal point. It's worth noting that the first-order necessary condition is straightforward to apply when $\M$ is a submanifold of a Euclidean space. In this section, we derive the second-order necessary optimality condition. Additionally, we present a sufficient condition for the nonsingularity of $J(x_*)$ in~\eqref{equ:J_x_*}.

Without loss of generality, let  $x_* = \begin{bmatrix}
\bar{x}_* \\ 0
\end{bmatrix}$ be a stationary point of~\eqref{1-1}, where $\bar{x}_* \in \R^j$. By Lemma~\ref{lemm-3-3}, it holds that
\[
M_{x_*} = \begin{bmatrix}
 \I_j & \\
 & 0_{n-j} 
\end{bmatrix}.
\]
By the partition of $M_{x_*}$, we have
\begin{equation*}
\begin{aligned}
B_{x_*} = \begin{bmatrix}
\bar{B}_{x_*} \\ \hat{B}_{x_*}
\end{bmatrix}, &\ \
\nabla^2 f(x_*) = \begin{bmatrix}
 H_{x_*}^{(11)} & H_{x_*}^{(12)}\\
 H_{x_*}^{(21)} & H_{x_*}^{(22)} 
\end{bmatrix}, & \mathcal{L}_{x_*} = \begin{bmatrix}
 L_{x_*}^{(11)} & L_{x_*}^{(12)}\\
 L_{x_*}^{(21)} & L_{x_*}^{(22)} 
\end{bmatrix},
\end{aligned}
\end{equation*}
where $\bar{B}_{x_*} \in \mathbb{R}^{j\times d}$ with $j\ge d$ has full column rank by Assumption~\ref{assu-3-1}, $H_{x_*}^{(11)} \in \mathbb{R}^{j\times j}$, $L_{x_*}^{(11)} \in \mathbb{R}^{j\times j}$. Therefore,

\begin{equation}\label{equ:J_x_*}
\begin{aligned}
J(x_*) &= -\left[\I_n - \Lambda_{x_*} + t\Lambda_{x_*} (\nabla^2 f(x_*) - \mathcal{L}_{x_*})\right] \\
&= -\begin{bmatrix}
\bar{B}_{x_*} \bar{B}_{x_*}^{\dagger} + t(\I_j - \bar{B}_{x_*} \bar{B}_{x_*}^{\dagger})(H_{x_*}^{(11)}-L_{x_*}^{(11)}) & t(\I_j - \bar{B}_{x_*} \bar{B}_{x_*}^{\dagger})(H_{x_*}^{(12)}-L_{x_*}^{(12)})\\
0_{(n-j) \times j}& \I_{n-j}
\end{bmatrix},
\end{aligned}
\end{equation}
where $\dagger$ denotes generalized \textit{Moore-Penrose inverse} and $ \bar{B}_{x_*}^{\dagger} =  (\bar{B}_{x_*}^{\T}\bar{B}_{x_*})^{-1}\bar{B}_{x_*}^{\T}$. 


Since $x_* = \begin{bmatrix}
\bar{x}_* \\ 0
\end{bmatrix}$, the non-smooth optimization problem~\eqref{1-1} can be separated into the smooth part and non-smooth part at the sufficiently small neighborhood of $x_*$, which correspond to $\bar{x}_*$ and the zero part, respectively. By fixing the location of the zero part, the second-order necessary condition follows from only considering the smooth part of the optimization problem. To do this, we first analyze the property of the set consisting of the points on $ \M$ that keep zero part unchanged.

Lemma~\ref{lemm-3-5} shows that the set on $\M$ that keeps its location of zero part unchanged is an embedded submanifold of $\mathbb{R}^n$ under reasonable conditions. 

\begin{lemma}\label{lemm-3-5}
Let $\mathcal{M}$ be an embedded submanifold of Euclidean space $\R^n$. Let $q(x) = Q x$, where $Q = \left[0_{(n-j)\times j},\ \I_{n-j}\right]$, if
\begin{equation}
\N_{x_*} \mathcal{M} \cap \mathrm{range}(Q^{\T}) = \{0\},\label{3-13}
\end{equation}
 then $\mathcal{N} = \{x\in \mathcal{M}: q(x)=0\}\cap \Omega_{x_*}$ is an embedded submanifold of $\mathbb{R}^n$, where $\Omega_{x_*}$ is a sufficiently small neighbourhood of $x_*$ and $\mathrm{range}(Q^{\T})$ denotes the columns space of $Q^T$. Furthermore, the tangent space of $\mathcal{N}$ at $x \in \mathcal{N}$ is 
\[
\T_{x} \mathcal{N} = \left\{ u =\begin{bmatrix}
u_1\\0
\end{bmatrix}
: \bar{B}_{x}^{\T} u_1 = 0, u_1 \in \mathbb{R}^j \right\}.
\]
\end{lemma}

\begin{proof}
The result follows from transversality theory, see~\cite[Theorem~6.30(b)]{Lee_2012}. We give a detailed proof for the reader's convenience.

If $\mathcal{M}$ is an open subset of $\mathbb{R}^n$, then $\mathcal{N}$ is the intersection of an open set of $\mathbb{R}^n$ with a hyperplane $\{x \in \mathbb{R}^n : q(x) = 0\}$. Therefore, it is an embedded submanifold of $\mathbb{R}^n$.

If $\mathcal{M}$ is not an open subset of $\mathbb{R}^n$, then for $x_* \in \mathcal{M}$, there exists a local defining function $\phi: \mathcal{U} \to \mathbb{R}^d$ satisfying
 \begin{itemize}
     \item[(a)] $\mathcal{M} \cap \mathcal{U} = \phi^{-1}(0) = \{y\in \mathcal{U} : \phi(y) = 0\};$ and
     \item[(b)] $\mathrm{rank}\D \phi(x_*) = d$,
 \end{itemize}
where $\mathcal{U}$ is a neighborhood of $x_*$ in $\R^n$. Define
\[
\psi(x) = \begin{pmatrix}
\phi(x)\\
q(x)
\end{pmatrix} : \mathcal{U}\to \mathbb{R}^{d+(n-j)}.
\]
We have
\begin{equation*}
\begin{aligned}
\mathrm{rank}(\D \psi(x_*)) = \mathrm{rank}
\begin{pmatrix}
  \D \phi(x_*)\\
  \D q(x_*)
  \end{pmatrix}
= \mathrm{rank}
\begin{pmatrix}
B_{x_*}^T \\
Q
  \end{pmatrix}
= \mathrm{rank}\begin{pmatrix}
B_{x_*} & Q^T
\end{pmatrix}
= d + n - j,
\end{aligned}
\end{equation*}
where the last equation follows from~\eqref{3-13}.

Therefore, there exist a neighborhood $\tilde{\Omega}_{x_*}$ of $x_*$ in $\mathbb{R}^n$ such that $\D \psi(x)$ is full row rank. Let $\Omega_{x_*} = \mathcal{U} \cap \tilde{\Omega}_{x_*}$. 
We have
\[
\mathcal{N} \cap \Omega_{x_*} = \psi^{-1}(0),
\]
and
\[
\mathrm{rank} \D \psi(x) = d+(n-j).
\]
It follows that $\mathcal{N}$ is an embedded submanifold of $\mathbb{R}^n$ with  $\dim \T_x \mathcal{N} = j-d$, where $j\ge d$. Furthermore, for any $x\in \mathcal{N}$, we have
\[
\begin{aligned}
    \T_{x} \mathcal{N} = \mathrm{ker}(\D \psi(x)) &= \left\{u: \D \phi(x)[u] = 0, \D q(x)[u]=0 \right\}\\
    &= \left\{u: B_{x}^{\T} u = 0, q(u)=0 \right\}\\
    &= \left\{u=\begin{bmatrix}
       u_1\\u_2
    \end{bmatrix}: \bar{B}_{x}^{\T} u_1 = 0, u_2 = 0 \right\}.
\end{aligned}
\]
\end{proof}

\begin{remark}\label{con_3_21}
    Condition~\eqref{3-13} is implied by Assumption~\ref{assu-3-1}. Specifically, for any $\mathbf{z} \in \N_{x_*} \mathcal{M} \cap \mathrm{range}(Q^{\T})$, we have $\mathbf{z} = Q^{\T}\mathbf{s} = B_{x_*}\mathbf{t}$ for certain vectors $\mathbf{s}$ and $\mathbf{t}$. 
    Since $Q = \left[0_{(n-j)\times j},\ \I_{n-j}\right]$, the first $j$ entries of $\mathbf{z}$ are zeros. It follows that the first $j$ entries of $B_{x_*}\mathbf{t}$ are zeros, which implies $\bar{B}_{x_*}\mathbf{t} = 0$. Combining with Assumption~\ref{assu-3-1} yields $\mathbf{t} = 0$. Therefore, $\mathbf{z} = B_{x_*}\mathbf{t}$ must be $0$, which implies Condition~\eqref{3-13} holds.
\end{remark}

Now, we are ready to give a second-order necessary optimality condition for Problem~\eqref{1-1}. 

\begin{proposition}\label{sec_opt}
Suppose Assumption~\ref{assu-3-1} holds. 
If $x_* = \begin{bmatrix}
\bar{x}_* \\ 0
\end{bmatrix}$ is a local optimal point of Problem~\eqref{1-1} with $\bar{x}_* \in \R^j$, then $v(x_*) = 0$ and $H_{x_*}^{(11)} - L_{x_*}^{(11)} \succeq 0$ on the subspace
$\{w: \bar{B}_{x_*}^{\T}w = 0\}$. Furthermore, if $H_{x_*}^{(11)} - L_{x_*}^{(11)} \succ 0$ on the subspace
$\{w: \bar{B}_{x_*}^{\T}w = 0\}$, then $J(x_*)$ in~\eqref{equ:J_x_*} is nonsingular.
\end{proposition}

\begin{proof}
It follows from~\eqref{v_*} that $v(x_*) = 0$.
According to the KKT condition \eqref{3-1} at $x_*$, we have 
\[
x_* = \prox_{t h}\bigl( x_* - t [\nabla f(x_*) + B_{x_*} \lambda_* ] \bigr).
\]
It follows from~\eqref{eq:EProx} that
\[
\bar{x}_* + t \mu \sgn(\bar{x}_*) = \bar{x}_* - t[\nabla f(x_*)]_1 -t \bar{B}_{x_*}\lambda_*,
\]
where $\nabla f(x_*)^{\T} = \Big[[\nabla f(x_*)]_1^{\T},\ [\nabla f(x_*)]_2^{\T} \Big]$ with $[\nabla f(x_*)]_1\in \R^j$. Therefore, 
\begin{equation}\label{3-16}
-\bar{B}_{x_*}\lambda_* = \mu \sgn(\bar{x}_*) + [\nabla f(x_*)]_1.
\end{equation}

As mentioned in Remark~\ref{con_3_21}, Condition~\eqref{3-13} holds, hence $\mathcal{N}$ in Lemma~\ref{lemm-3-5}
is an embedded submanifold.
Let $\theta = -\hat{B}_{x_*}\lambda_* - [\nabla f(x_*)]_2 \in \R^{n-j}$. Define a function
\[
\Tilde{F}(x) = f(x) + \mu \mathrm{sgn}(\bar{x}_*)^{\T} x_1 + \theta^{\T} x_2,
\]
where $x = \begin{bmatrix}
x_1\\x_2
\end{bmatrix} \in \mathcal{N}$ with $x_1 \in \R^k$. 
Therefore, it holds that $F(x_*) =  \Tilde{F}(x_*)$, where $F(x)$ is the objection function in~\eqref{1-1}. Since $\nabla \Tilde{F}(x_*) = \nabla f(x_*) + \begin{bmatrix}
\mu \sgn(\bar{x}_*)\\ \theta
\end{bmatrix} $, it follows from~\eqref{3-16} that
\begin{equation}
   \nabla \Tilde{F}(x_*) = \begin{bmatrix}
[\nabla f(x_*)]_1 + \mu \sgn(\bar{x}_*)\\ [\nabla f(x_*)]_2 + \theta
\end{bmatrix} = \begin{bmatrix}
-\bar{B}_{x_*}\lambda_*\\-\hat{B}_{x_*}\lambda_*
\end{bmatrix} = - B_{x_*}\lambda_*. \label{3-14}
\end{equation}
Let $\gamma(t)$ be a smooth curve on $\mathcal{N}$ with $\gamma(0) = x_*$ and $\gamma^{\prime}(0) = u$ for any $u \in \T_{x_*} \mathcal{N}$. Therefore, $\gamma(t)$ is also a smooth curve on $\M$, It follows that $t=0$ is a local optimal point of $F(\gamma(t)) = \Tilde{F}(\gamma(t))$. Since $\tilde{F}(\gamma(t))$ is smooth at $t = 0$, we have
\[
\left.\frac{\mathrm{d}^2}{\mathrm{d} \mathrm{t}^2} \Tilde{F}(\gamma(t))\right|_{\mathrm{t}=0} \ge 0.
\]
Note that 
\[
\frac{\mathrm{d}}{\mathrm{d} \mathrm{t}} \Tilde{F}(\gamma(t)) = \left\langle \grad \Tilde{F}(\gamma(t)), \gamma^{\prime}(t)\right\rangle,
\]
and
\[
\left.\frac{\mathrm{d}^2}{\mathrm{d} \mathrm{t}^2} \Tilde{F}(\gamma(t))\right|_{\mathrm{t}=0} = \left\langle \Hess \Tilde{F}(x_*)[u], u\right\rangle + \left\langle \grad \Tilde{F}(x_*), \gamma^{\prime\prime}(0)\right\rangle.
\]
Since $x_*$ is  a local optimal point of $\Tilde{F}(x)$ over the manifold $\mathcal{N}$, so that $\grad \Tilde{F}(x_*) = 0$. On the other hand, for any $u \in \T_{x_*}\mathcal{N}$, we have
\begin{equation*}
\begin{aligned}
    \Hess &\Tilde{F}(x_*)[u] = \Proj_{\T_{x_*}\mathcal{N}} \nabla^2 \Tilde{F}(x_*)[u] + \Tilde{\mathcal{W}}_{x_*}\left(u, \Proj_{\T_{x_*}\mathcal{N}}^{\perp}\nabla \Tilde{F}(x_*)\right)\\
    &\overset{(i)}{=} \Proj_{\T_{x_*}\mathcal{N}} \nabla^2 \Tilde{F}(x_*)[u] + \Tilde{\mathcal{W}}_{x_*}\left(u, \begin{bmatrix}
    \bar{B}_{x_*} & \\
    & \I_{n-j}
    \end{bmatrix}
    \begin{bmatrix}
    \bar{B}_{x_*} & \\
    & \I_{n-j}
    \end{bmatrix}^{\dagger}
    \nabla \Tilde{F}(x_*)\right)\\
    &\overset{(ii)}{=} \Proj_{\T_{x_*}\mathcal{N}} \nabla^2 \Tilde{F}(x_*)[u] - \Tilde{\mathcal{W}}_{x_*}\left(u, B_{x_*}\lambda_* \right)
\end{aligned}
\end{equation*}
where $\Tilde{\mathcal{W}}_{x_*}$ denotes the Weingarten map on manifold $\mathcal{N}$,
$(i)$ follows from $ \begin{bmatrix}
    \bar{B}_{x_*} & \\
    & \I_{n-j}
    \end{bmatrix}$ is a basis of $\N_{x_*} \mathcal{N}$, $(ii)$ follow from~\eqref{3-14}.

Note that  $ \begin{bmatrix}
    \bar{B}_{x} & \\
    & \I_{n-j}
    \end{bmatrix}$ is a basis of $\N_{x} \mathcal{N}$, then for any $u=\begin{bmatrix}
u_1 \\ u_2
\end{bmatrix}\in \T_{x_*}\mathcal{N}$ satisfying $\bar{B}_{x_*}^{\T}u_1 = 0$ and $u_2 = 0$, we have
\begin{equation*}
    \begin{aligned}
        \Tilde{\mathcal{W}}_{x_*}\left(u, B_{x_*}\lambda_* \right) &\overset{\eqref{2-3}}{=} \D\left(I_n - \begin{bmatrix}
    \bar{B}_{x} & \\
    & \I_{n-j}
    \end{bmatrix}
    \begin{bmatrix}
    \bar{B}_{x} & \\
    & \I_{n-j}
    \end{bmatrix}^{\dagger}\right)[u] \cdot B_{x_*}\lambda_* \\
    &= -\begin{bmatrix}
    \D\bar{B}_{x} & \\
    & 0_{n-j}
    \end{bmatrix}
    \begin{bmatrix}
u_1 \\ u_2
\end{bmatrix}
    \cdot
     \begin{bmatrix}
    \bar{B}_{x_*}^{\dagger} & \\
    & \I_{n-j}
    \end{bmatrix}
     \begin{bmatrix}
        \bar{B}_{x_*}\lambda_*  \\
    \hat{B}_{x_*}\lambda_*
    \end{bmatrix} \\
    &\quad- \begin{bmatrix}
    \bar{B}_{x_*} & \\
    & I_{n-j}
    \end{bmatrix}\cdot
     \begin{bmatrix}
    \D\bar{B}_{x}^{\dagger} & \\
    & 0_{n-j}
    \end{bmatrix}\begin{bmatrix}
u_1 \\ u_2
\end{bmatrix}\cdot
    \begin{bmatrix}
        \bar{B}_{x_*}\lambda_*  \\
    \hat{B}_{x_*}\lambda_*
    \end{bmatrix}
    \\
     &\overset{(a)}{=}-\Proj_{\T_{x_*}\mathcal{N}}\left(\begin{bmatrix}
    \D\bar{B}_{x} & \\
    & 0_{n-j}
    \end{bmatrix}
    \begin{bmatrix}
u_1 \\ u_2
\end{bmatrix}
    \cdot
     \begin{bmatrix}
    \bar{B}_{x_*}^{\dagger} & \\
    & \I_{n-j}
    \end{bmatrix}
     \begin{bmatrix}
        \bar{B}_{x_*}\lambda_*  \\
    \hat{B}_{x_*}\lambda_*
    \end{bmatrix}\right)\\
    &=-\Proj_{\T_{x_*}\mathcal{N}}\left(
   \begin{bmatrix}
       \D \bar{B}_x[u_1]\cdot \lambda_*\\
       0_{(n-j)\times 1}
   \end{bmatrix}
    \right),
    \end{aligned}
\end{equation*} 
where (a) follows from $\Tilde{W}_{x_*}(u, B_{x_*}\lambda_{x_*}) = \Proj_{\T_{x_*}\mathcal{N}}\left(\Tilde{W}_{x_*}(u, B_{x_*}\lambda_{x_*}) \right)$. Therefore,
\begin{equation}\label{3-32}
    \left\langle u, \Tilde{\mathcal{W}}_{x_*}(u, B_{x_*}\lambda_*)\right\rangle = -u_1^{\T}  \D \bar{B}_x[u_1]\cdot \lambda_*.
\end{equation}
Furthermore, for any $\eta \in \T_{x_*}\M$ satisfying $B_{x_*}^{\T}\eta = 0$, we have
$$
\begin{aligned}
\mathcal{W}_{x_*}(\eta,B_{x_*}\lambda_*) &= \D(\I_n - B_{x}B_{x}^{\T})[\eta]\cdot B_{x_*}\lambda_*\\
&= -\D B_x[\eta]\cdot B_{x_*}^{\T}B_{x_*}\lambda_* - B_{x_*}\cdot \D B_{x}^{\T} [\eta]\cdot B_{x_*}\lambda_*\\
&=-\Proj_{\T_{x_*}\M} \D B_x[\eta]\cdot \lambda_*,
\end{aligned}
$$
where $\mathcal{W}_{x_*}$ denotes the Weingarten map on manifold $\mathcal{M}$. We take $\eta = u \in \T_{x_*}\mathcal{N} \subset \T_{x_*}\M$, then
\begin{equation*}
\begin{aligned}
\left\langle u, \mathcal{W}_{x_*}(u, B_{x_*}\lambda_*)\right\rangle &= -u^{\T}  \D B_x[u]\cdot \lambda_*\\
&= -u_1^{\T}  \D \bar{B}_x[u_1]\cdot \lambda_*\\
&\overset{\eqref{3-32}}{=}\left\langle u, \Tilde{\mathcal{W}}_{x_*}(u, B_{x_*}\lambda_*)\right\rangle
\end{aligned}
\end{equation*}
Therefore, for any $u=\begin{bmatrix}
u_1 \\ u_2
\end{bmatrix}\in \T_{x_*}\mathcal{N}$ satisfying $\bar{B}_{x_*}^{\T}u_1 = 0$ and $u_2 = 0$, we have 
\begin{equation*}
\begin{aligned}
    0 \le \left\langle \Hess\Tilde{F}(x_*)[u], u\right\rangle &= \left[u_1^{\mathrm{T}}, 0\right] \nabla^2 f(x_*)\begin{bmatrix}
    u_1\\0
    \end{bmatrix} - [u_1^{\T}, 0] \Tilde{\mathcal{W}}_{x_*}\left(\begin{bmatrix}
    u_1\\0
    \end{bmatrix}, B_{x_*}\lambda_* \right)\\
     &= \left[u_1^{\mathrm{T}}, 0\right] \nabla^2 f(x_*)\begin{bmatrix}
    u_1\\0
    \end{bmatrix} - [u_1^{\T}, 0] \mathcal{W}_{x_*}\left(\begin{bmatrix}
    u_1\\0
    \end{bmatrix}, B_{x_*}\lambda_* \right)\\
    &=  u_1^{\T} H_{x_*}^{(11)} u_1 - u_1^{\T} L_{x_*}^{(11)} u_1.
\end{aligned}
\end{equation*}
Therefore, on the subspace $\{w: \bar{B}_{x_*}^{\T}w = 0\}$, we have
\[
H_{x_*}^{(11)}- L_{x_*}^{(11)} \succeq 0.
\]
Furthermore, if $H_{x_*}^{(11)}- L_{x_*}^{(11)} \succ 0$ on the subspace
$\{w: \bar{B}_{x_*}^{\T}w = 0\}$, we have $J(x_*): \T_{x_*}\M \to \T_{x_*}\M$, i.e.,
\[
\begin{aligned}
J(x_*) = -\begin{bmatrix}
\bar{B}_{x_*} \bar{B}_{x_*}^{\dagger} + t(\I_j - \bar{B}_{x_*} \bar{B}_{x_*}^{\dagger})(H_{x_*}^{(11)}- L_{x_*}^{(11)}) & t(\I_j - \bar{B}_{x_*} \bar{B}_{x_*}^{\dagger})(H_{x_*}^{(12)}- L_{x_*}^{(12)})\\
0_{(n-j) \times j}& \I_{n-j}
\end{bmatrix},
\end{aligned}
\]
where $\bar{B}_{x_*}^{\dagger} =  (\bar{B}_{x_*}^{\T}\bar{B}_{x_*})^{-1}\bar{B}_{x_*}^{\T}$. In order to verify $J(x_*)$ is nonsingular, we only need to verify $J_{x_*}^{(11)} = \bar{B}_{x_*} \bar{B}_{x_*}^{\dagger} + t(\I_j - \bar{B}_{x_*} \bar{B}_{x_*}^{\dagger})(H_{x_*}^{(11)}- L_{x_*}^{(11)})$ is nonsingular. For any $\eta \in \R^{j}$, we have
\[
J_{x_*}^{(11)} \eta = \underbrace{\bar{B}_{x_*} \bar{B}_{x_*}^{\dagger} \eta}_{s_1} + \underbrace{t(\I_j - \bar{B}_{x_*} \bar{B}_{x_*}^{\dagger}) (H_{x_*}^{(11)}- L_{x_*}^{(11)})\eta}_{s_2},
\]
where $s_1$  is orthogonal to $s_2$. Thus, if $J_{x_*}^{(11)}\eta = 0$, then $s_1 = 0$ and $s_2 = 0$. Since $s_1 = \bar{B}_{x_*} \bar{B}_{x_*}^{\dagger} \eta = 0$, then $\eta = (\I_j - \bar{B}_{x_*} \bar{B}_{x_*}^{\dagger}) \eta$. According to $s_2 = 0$, we have 
\[
\begin{aligned}
0 = s_2 &= t(\I_j - \bar{B}_{x_*} \bar{B}_{x_*}^{\dagger}) (H_{x_*}^{(11)}- L_{x_*}^{(11)})\eta\\
&= t(\I_j - \bar{B}_{x_*} \bar{B}_{x_*}^{\dagger})(H_{x_*}^{(11)}- L_{x_*}^{(11)}) (\I_j - \bar{B}_{x_*} \bar{B}_{x_*}^{\dagger}) \eta,
\end{aligned}
\]
then $0 = \eta^{\T} s_2 = t\eta^{\T}(\I_j - \bar{B}_{x_*} \bar{B}_{x_*}^{\dagger})(H_{x_*}^{(11)}- L_{x_*}^{(11)}) (\I_j - \bar{B}_{x_*} \bar{B}_{x_*}^{\dagger}) \eta$. With the condition $H_{x_*}^{(11)}- L_{x_*}^{(11)} \succ 0$ on the subspace
$\{\omega: \bar{B}_{x_*}^{\T}\omega = 0\}$, we have $\eta = 0$. Therefore, $J_{x_*}^{(11)}\eta = 0$ if and only if $\eta = 0$, which implies $J_{x_*}^{(11)}$ is nonsingular. It follows that we have $J(x_*)$ is nonsingular.
\end{proof}

\section{Globalization}\label{sec:glo}
In Section~\ref{sec:RPN}, the Riemannian proximal Newton method in Algorithm~\ref{alg:RPN} is guaranteed to converge locally and quadratically. However, it may not converge globally. In this section, a hybrid version by concatenating the Riemannian proximal gradient method in~\cite{2020Proximal} and Algorithm~\ref{alg:RPN} is given and stated in Algorithm~\ref{alg:RPN_s}.

Specifically, if the norm of the search direction $v_k$ is not sufficiently small, then the search direction is taken to be the proximal gradient step, see Steps from~\ref{alg:RPN_s:st01} to~\ref{alg:RPN_s:st02}, which has global convergence property and is well-defined~\cite[Lemma 5.2]{2020Proximal}. Otherwise, the proximal Newton step is used. Next, we show that if $\|v_k\|$ is sufficiently small and certain reasonable assumptions hold, then the iterate $x_k$ is sufficiently close to a local minimizer $x_*$. Therefore, the local quadratic convergence rate follows from Theorem~\ref{the_super}.

\begin{algorithm}
\caption{A globalized version of the Riemannian proximal Newton method (RPN-G)}\label{alg:RPN_s}
\begin{algorithmic}[1] 
\Require
$x_0 \in \mathcal{M}$, $t > 0$, $\rho\in (0,\frac12]$,
$\epsilon > 0$
\State $k = 0$;
\Loop 
\State Compute $v_k$ by solving~\eqref{eq:subforv} with $t$;
\If { $\|v_k\| \leq \epsilon$ }
\State $K = k$ and break;
\EndIf
\State Set $\alpha = 1$;\label{alg:RPN_s:st01}
\While{$F(R_{x_k}(\alpha v_k)) > F(x_k) - \frac{1}{2}\alpha\|v_k\|^2$}  
\State $\alpha = \rho \alpha$;
\EndWhile 
\State $x_{k+1} = R_{x_k}(\alpha v_k)$; \label{alg:RPN_s:st02}
\State $k = k + 1$;
\EndLoop 
\For{$k = K+1,K+2,\dots$}
\State Compute $v_{k}$ by solving~\eqref{eq:subforv} with $t$;
\State Compute $u_{k}$ by solving~\eqref{3-3};
\State  $x_{k+1} = R_{x_{k}}(u_{k})$;
\EndFor
\end{algorithmic}
\end{algorithm}

\begin{theorem}\label{the_4-2}
 Suppose that 
    \begin{itemize}
        \item[(a)] $\M$ is a compact Riemannian submanifold of $\R^n$, the retraction $R$ is globally defined, $f:\R^n \to \R$ has Lipschitz continuous gradient with Lipschitz constant $L$ in the convex hull of $\M$, i.e., $\|\nabla f(x) - \nabla f(y)\|_{\F} \leq L \|x - y\|_{\F}$ for all $x, y$ in the convex hull of $\M$;
    \end{itemize} 

    Let $\{\hat{x}_k\}$ denote the sequence generated by Algorithm~\ref{alg:RPN_s} with $\epsilon = 0$ and $x_*$ denote an accumulation point of $\{\hat{x}_k\}$.
Then  $x_*$ is a stationary point. Further, suppose that
    \begin{itemize}
        \item[(b)] $x_*$ is a strict local optimal point and isolated stationary point of $F$, i.e., there exists a neighborhood $\mathcal{N}_{x_*}$ of $x_*$ such that for any $y \in \mathcal{N}_{x_*}$ and $y \neq x_*$, it holds that $F(y) > F(x_*)$ and $x_*$ is the only stationary point in $\mathcal{N}_{x_*}$.
        \item[(c)] The assumptions of Theorem~\ref{the_super} hold at $x_*$ (i.e., Assumption~\ref{assu-3-1} and Assumption~\ref{assu-3-2} hold, and $J(x_*)$ is nonsingular).
    \end{itemize}
    
    Then there exists $\epsilon >0$, such that the iterates $\{x_k\}$ generated by Algorithm~\ref{alg:RPN_s} with the $\epsilon > 0$ converges to the local optimal $x_*$ and the local convergence rate is quadratic.
\end{theorem}
\begin{proof}
Since $\epsilon = 0$, then Algorithm~\ref{alg:RPN_s} degenerates to ManPG~\cite{2020Proximal}. It follows from (a) and~\cite[Theorem~5.5]{2020Proximal} that $\lim_{k\to \infty}\|v_k\| = 0$ and $x_*$ is a stationary point of~\eqref{1-1}.  
    
Assume (b) and (c) hold. Next, by an argument akin to the capture theorem~\cite[Theorem~4.4.2]{absil2009}, we show that $\{\hat{x}_k\}$ converges to $x_*$. Specifically, since $x_*$ is a strict local optimal point of $F$, there exists $\delta > 0$ such that for any $y \in \mathbb{B}(x_*, \delta) := \{ z \in \mathcal{M} \mid \mathrm{dist}(x_*, z) \leq \delta \}$ and $y \neq x_*$, it holds that $F(y) > F(x_*)$. Let $\mathcal{B}(x_*, \delta) = \{ z \in \mathcal{M} \mid \mathrm{dist}(x_*, z) = \delta \}$ denote the boundary of $\mathbb{B}(x_*, \delta)$ and define $F_{\delta, x_*} = \inf_{y \in \mathcal{B}(x_*, \delta) } F(y)$. Since $\mathcal{B}(x_*, \delta)$ is compact, it holds that $F_{\delta, x_*} > F(x_*)$. Let $\mathcal{N}_{x_*}$ denote $\mathbb{B}(x_*, \delta) \cap \{z \in \mathcal{M} \mid F(z) < F(x_*) + \left(F_{\delta, x_*} - F(x_*)\right) / 2 \}$. Therefore, it holds that $\mathrm{dist}(\mathcal{B}(x_*, \delta), \mathcal{N}_{x_*}) := \inf_{z_1 \in \mathcal{B}(x_*, \delta), z_2 \in \mathcal{N}_{x_*}} \mathrm{dist}(z_1, z_2) > 0$, otherwise, it violates the continuity of $F$. Since $x_*$ is an accumulation point of $\{\hat{x}_k\}$ and $\|v_k\| \rightarrow 0$, there exists $K > 0$ such that for all $k > K$, inequality $\|v_k\| < \mathrm{dist}(\mathcal{B}(x_*, \delta), \mathcal{N}_{x_*})$ holds, and there exists $J > K$ such that $\hat{x}_J \in \mathcal{N}_{x_*}$. Since ManPG is a descent algorithm in the sense that $F(\hat{x}_k) > F(\hat{x}_{k+1})$ for all $k$, we have $\hat{x}_k \in \mathcal{N}_{x_*}$ for all $k \geq J$. Since $x_*$ is the only stationary point of $F$ in $\mathcal{N}_{x_*}$, $\hat{x}_k$ converges to~$x_*$.
   
   Since $\hat{x}_k$ converges to $x_*$, there exists $\bar{K} > 0$ such that for all $k > \bar{K}$, $\hat{x}_k \in \mathcal{V}$, where $\mathcal{V}$ is the neighborhood of $x_*$ defined in Theorem~\ref{the_super}. Define $\Delta_k = \min_{j \leq k} \|v_j\|$. If $\Delta_k = 0$, then there exists $j \leq k$ such that $v_j = 0$. By the definition of Algorithm~\ref{alg:RPN_s}, we have $x_{j} = x_{j + 1} = \ldots$. Since $x_*$ is the limit point of $\{x_k\}$, it holds that $x_j = x_{j + 1} = \ldots = x_*$, which implies finite step termination. Suppose that $\Delta_k > 0$ for all $k$. Let $\epsilon < \Delta_{\bar{K}}$. It follows that $x_k = \hat{x}_k$ for all $k \leq \bar{K}$ and $\{x_k\}_{k \geq \bar{K}}$ converges to $x_*$ quadratically by Theorem~\ref{the_super}.
\end{proof}

\begin{remark}
    The global and local quadratic convergence of Algorithm~\ref{alg:RPN_s} by Theorem~\ref{the_4-2} requires 
    a sufficiently small $\epsilon$. However, if the value of $\epsilon$ is too small, then it is hard to observe quadratic convergence behavior. In the numerical experiments, 
    $\epsilon$ is chosen empirically.
\end{remark}

\section{Numerical Experiments}\label{sec:Num}

In this section, we perform numerical experiments to test the efficiency of our proposed RPN for solving the sparse PCA given as
\begin{equation}\label{spca:1}
\min_{X\in \mathrm{St}(n,r)} -\trace(X^{\T} A^{\T} A X) + \mu \|X\|_1,
\end{equation}
where $\mathrm{St}(n,r) = \{X\in \R^{n\times r}: X^{\T}X = \I_r\}$ is the Stiefel manifold, $A\in \R^{m\times n}$ is the data matrix, $\mu >0$ is the sparsity inducing regularization parameter.
The main benefit of Sparse PCA arises in high-dimensional setting, when, due to the introduced sparsity in the principal vectors, it remains a consistent estimator, while the classical PCA fails~\cite{jolliffe2003}. The optimization problem arising from sparse PCA has been used as a benchmark to test other Riemannian proximal methods, such as~\cite{2020Proximal,Huang2022a}.   
All experiments are performed in MATLAB R2022b on a standard PC with 2.8 GHz CPU (Inter Core i7).

We first numerically demonstrate in Subsection~\ref{subsec:naive} that the naive second order generalization, which comes from simply replacing the Euclidean gradient and Hessian with their Riemannian counterparts, does not achieve local superlinear convergence. In Subsection~\ref{subsec:loc_sup}, we verify our theoretical results showing that RPN does achieve quadratic convergence, and finally, in Subsection~\ref{subsec:compare} and Subsection~\ref{subsec:compare2} we compare its performance against ManPG.

\subsection{Naive Generalization}\label{subsec:naive}

The naive analogue in the Riemannian setting is constructed by simply replacing the Euclidean gradient and Hessian of~\eqref{1-3} by its Riemannian counterparts:
\begin{equation}\label{Rie_naive}
\begin{cases}v_k = \argmin_{v \in \T_{x_k} \M}\ f(x_k) + \left\langle\grad f\left(x_k\right), v\right\rangle+\frac{1}{2}\left\langle v, \Hess f(x_k) [v]\right\rangle + h\left(x_k+v\right), \\ x_{k+1}= R_{x_k}(v_k).\end{cases}
\hspace{-0.5em}
\end{equation}
We construct a simple example to show that \eqref{Rie_naive} generally does not yield a superlinear convergence when applied to the simple case of sparse PCA with $r=1$ defined over a unit sphere, i.e.,
\begin{equation}\label{spca:0}
\min_{x\in \mathbb{S}^{n-1}} -x^{\T} A^{\T} A x + \mu \|x\|_1.
\end{equation}
Essentially, the goal of \eqref{spca:0} is to compute the loading eigenvectors of $A^{\T} A$ while promoting sparsity. 

We handcraft the data matrix $A$ in the following way. Let $\Sigma = \diag\{[20, 0.1, 0.05]\}$ such that $\Sigma^2$ is the matrix of nonzero eigenvalues values of $A^TA$ and its corresponding eigenvectors $U$ are defined as
\[
 U = \begin{bmatrix}
     \I_3\\ 0_3
 \end{bmatrix}\in \R^{6\times 3},
\]
where $\I_3$ is $3\times 3$ identity matrix and $0_3$ is the $3\times 3$ matrix of zeros. Constructing $A = \Sigma U^{\T} \in \R^{3\times 6}$ results in
\[
 A^{\T} A = \begin{bmatrix}
     \Sigma^2 & 0_3\\
     0_3& 0_3
 \end{bmatrix},
\]
making an optimal solution of~\eqref{spca:0} to be $x_* = [1,0,0,0,0,0]^{\T}$. According to~\eqref{eq:Rhess}, for any $\eta_{x_*}\in \T_{x_*}\mathbb{S}^{n-1}$,
\[
 \begin{aligned}
    \eta_{x_*}^{\T} \Hess f(x_*)[\eta_*] &= \eta_{x_*}^{\T}\left[\Proj_{x_*} \left(\nabla^2 f(x_*) [\eta_{x_*}]\right) + \W_{x_*}\left(\eta_{x_*}, \Proj_{x_*}^{\perp} \big(\nabla f(x_*)\big) \right)\right]\\
    &=\eta_{x_*}^{\T}\left[(\I_6 - x_* x_*^{\T}) (-2 A^{\T}A \eta_{x_*}) + 2(x_*^{\T} A^{\T} A x_*) \eta_{*}\right]\\
    &= 2 \eta_{x_*}^{\T} \left((x_*^{\T}A^{\T} A x_*) \I_6 - A^{\T} A 
    \right)\eta_{x_*}^{\T}.
 \end{aligned}
\]
It can easily be verified that $\Hess f(x_*)$ is positive definite on the tangent space $\T_{x_*} \mathbb{S}^{n-1}$.
 
Since the positive definiteness is a local property, we add a sufficiently small perturbation to $A$, e.g., $A = A + 0.1 * R_1$, and select the initial point $x_0 = x_{*} + 0.1 * R_2$,  where the entries of 
 $R_1\in \R^{3\times 6}$ and $R_2 \in \R^{6\times 1}$ are sampled from the standard normal distribution $\mathcal{N}(0,1)$. In our experiments, we solve $v_k$ in \eqref{Rie_naive} by the semismooth Newton method similar to ManPG algorithm in~\cite{2020Proximal}.
We use the retraction $R_x(\eta_x) = (x+\eta_x)/\|x+\eta_x\|$ on the sphere manifold, where $x\in \mathbb{S}^{n-1}$, $\eta_x \in \T_x \mathbb{S}^{n-1}$. We compare the naive generalization, denoted RPN-N, with ManPG and our algorithm RPN in Figure~\ref{fig:naive}.
We clearly see in the Figure~\ref{fig:naive} that 
the naive generalization (RPN-N) fails to achieve the superlinear rate of convergence, exhibiting instead linear convergence, while our proposed RPN does. 

\begin{remark}
The naive generalization can be viewed as a specific instance of~\eqref{manpqn}, which was shown to have local linear convergence in~\cite{Wang2023}. On the other hand, a possible explanation why RPN-N seems to only achieve linear convergence rate is that it only considers the first-order approximation of $R_x(v)$ in $h(R_x (v))$, i.e., $h(x+v)$. However, to perform a second-order accurate approximation of $R_x(v)$ or $h(R_x(v))$ would be a computationally expensive task.
 \end{remark}
 
\begin{figure}
\centering 
\includegraphics[width=0.5\textwidth]{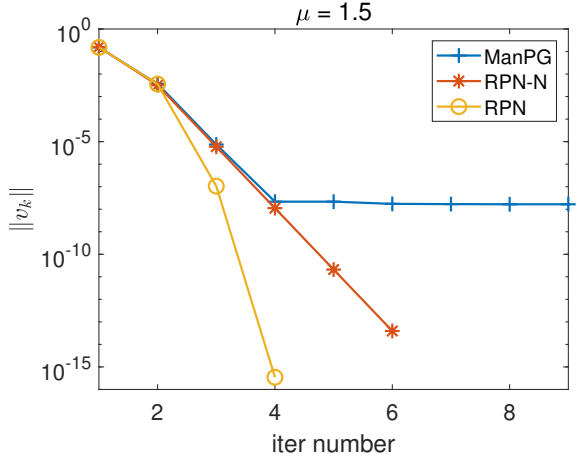}
\caption{Comparisons with RPN and ManPG algorithms.}\label{fig:naive} 
\end{figure}

\subsection{Local quadratic convergence of Algorithm~\ref{alg:RPN}}\label{subsec:loc_sup}

We proceed by testing the local quadratic convergence rate for different problem sizes in sparse PCA. For this task, we will use the polar retraction~\cite[Example~4.1.3]{absil2009} defined as $
R_x(\eta_x) = (x + \eta_x)(\I_r + \eta_x^{\T} \eta_x)^{-1/2}$, where $x\in \mathrm{St}(n,r)$, $\eta_x\in \T_x \mathrm{St}(n,r)$. We generate the random data matrix $A\in \R^{m\times n}$ such that its entries are drawn from the standard normal distribution $\mathcal{N}(0,1)$. In Algorithm~\ref{alg:RPN}, we set $m=50$ and use $t = 1/(2\|A\|_2^2)$, which corresponds to the choice of the stepsize in~\cite{2020Proximal,Huang2022a}.

In order to observe the local quadratic convergence, we first choose an initial point $x_0$ that is sufficiently close $x_*$, where $x_*$ is a stationary point of \eqref{1-1}.
Theorem~\ref{the_4-2} shows that
$x_k$ is sufficiently close to $x_*$ when $v_k$ is sufficiently small, so we first run the ManPG algorithm to choose a point $x_k$ satisfying $\|v_k\|_{\F} \le 10^{-4}$ as the initial point $x_0$ of Algorithm~\ref{alg:RPN}. 
Figure~\ref{Fig:1} shows $\|v_k\|$ versus the number of iterations for multiple values of $n$, $r$, and $\mu$. We can see that the proposed method RPN empirically shows quadratic convergence results which is consistent with the theoretical result.

\begin{figure}[H]
\centering 
\includegraphics[width=0.325\textwidth]{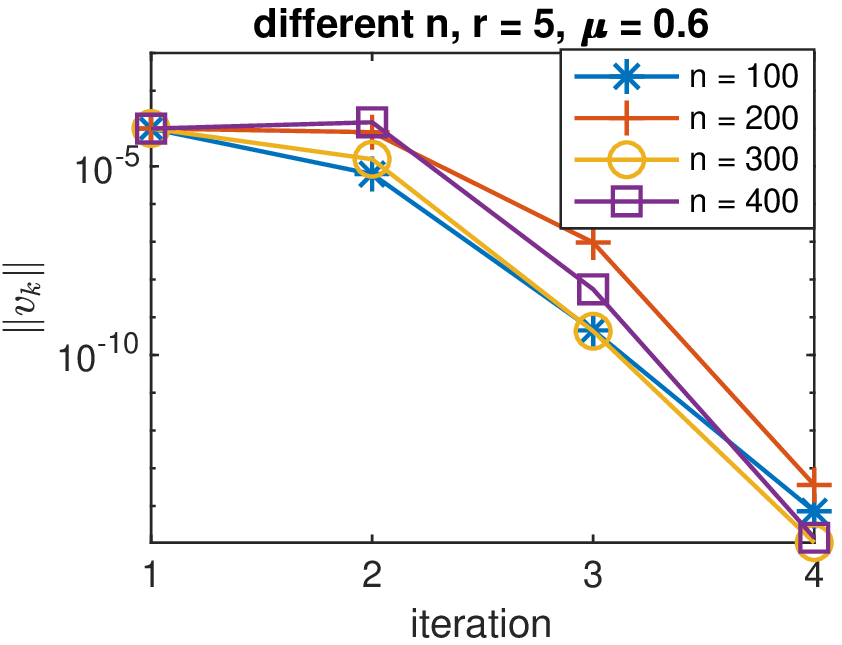}
\includegraphics[width=0.325\textwidth]{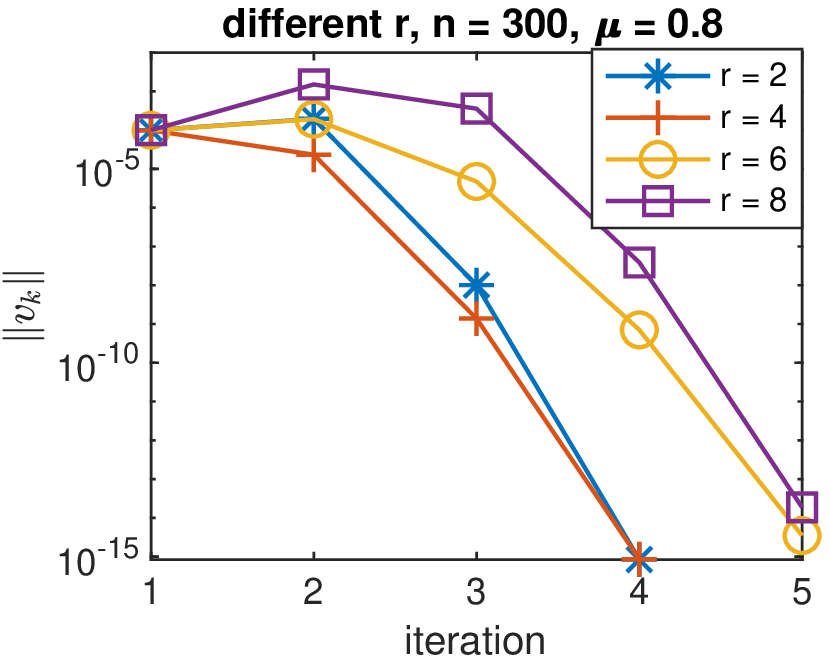}
\includegraphics[width=0.325\textwidth]{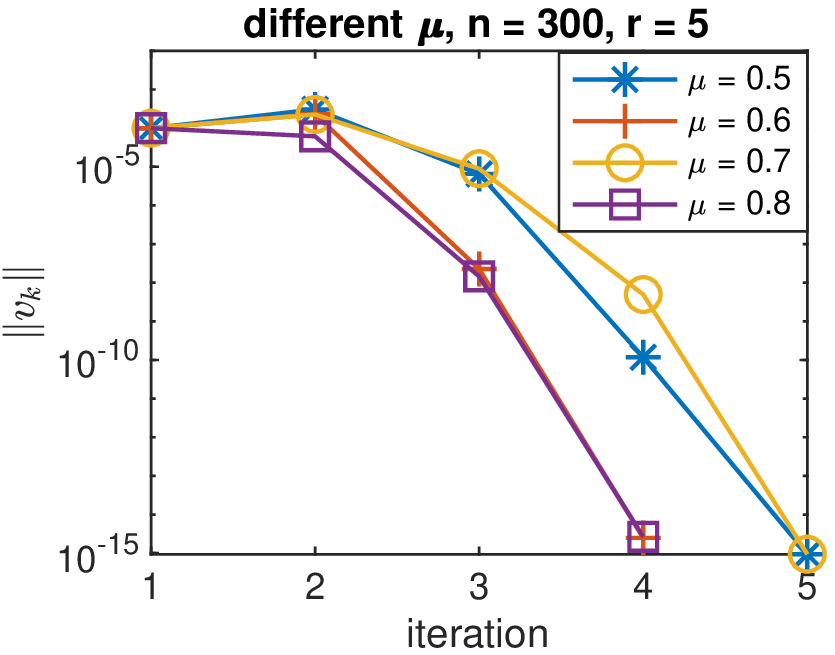}
\label{fig:1} 
\caption{ 
Random data: the norm of search direction for the SPCA. Left: different $n = \{100,200,300,400\}$ with $r = 5$ and $\mu = 0.6$; Middle:different $r = \{2,4,6,8\}$ with $n = 300$ and $\mu = 0.8$; Right: different $\mu = \{0.5,0.6,0.7,0.8\}$ with $n = 300$ and $r = 5$} \label{Fig:1}
\end{figure}

\subsection{Comparisons with ManPG algorithm on the sphere ($r = 1$)}\label{subsec:compare}

In this section, we consider the version of sparse PCA with $r = 1$, which admits a simple computation of $u_k$ by solving the linear equations~\eqref{3-3}. Thus
\eqref{spca:1} is simplified as 
\begin{equation}\label{spca:2}
\min_{x\in \mathbb{S}^{n-1}} -x^{\T} A^{\T} A x + \mu \|x\|_1,
\end{equation}
where the Stiefel manifold reduces to the unit sphere, as discussed in~\cite{jolliffe2003,d2004direct}.
We compare RPN-G, as stated in Algorithm~\ref{alg:RPN_s}, with the proximal Riemannian gradient method ManPG in~\cite{2020Proximal}.

The parameters are set as those in Subsection~\ref{subsec:loc_sup}. $\epsilon$ is set to be $10^{-4}$. Linear system~\eqref{3-3} is solved by the built-in Matlab function \textit{cgs}.
ManPG does not terminate until the number of iterations attains the maximal iteration (3000). RPN-G does not terminate until $\|v_k\| \le 10^{-12}$.  

\textbf{Random data:} The matrix $A$ is generated as that in Subsection~\ref{subsec:loc_sup}. The results with multiple values of $n$ and $\mu$ are reported in Table~\ref{tab:1} and Figure~\ref{fig:cpu}. Compared to ManPG, the proposed method RPN-G is able to obtain higher accurate solutions in the sense that $\|v_k\|$ from RPN-G is in double precision whereas that from ManPG is only in single precision. In addition, the number of Riemannian proximal Newton steps is not large and usually is only 5-6 iterations. It follows that the RPN-G is faster than ManPG when a highly accurate solution is needed, as shown in Figure~\ref{fig:cpu}.

\textbf{Synthetic data:} The comparisons are repeated for synthetic data, which is generated by following the experiments in \cite{Huang2022a,SpaSM}. Specifically, the five principal components shown in Figure~\ref{figPCs} are repeated $m/5$ times to obtain $m \times n$ noise-free matrix. Then the data matrix $A$ is created by adding each entry of the noise-free matrix by i.i.d. random value drawn from $\mathcal{N}(0,0.25)$. We set $m = 400$, $n= 4000$ and $\mu = 1.2$. 
The left and right plots in Figure~\ref{fig:sy_data} show $\|v_k\|$ versus the number of iterations and $\|v_k\|$ versus CPU time, respectively. The behavior of RPN-G and ManPG on the synthetic data is similar to that on the random data, that is, RPN-G converges faster in the sense of both computational time and the number of iterations.

\begin{figure}[ht!]
\centering
\includegraphics[width=0.7\textwidth]{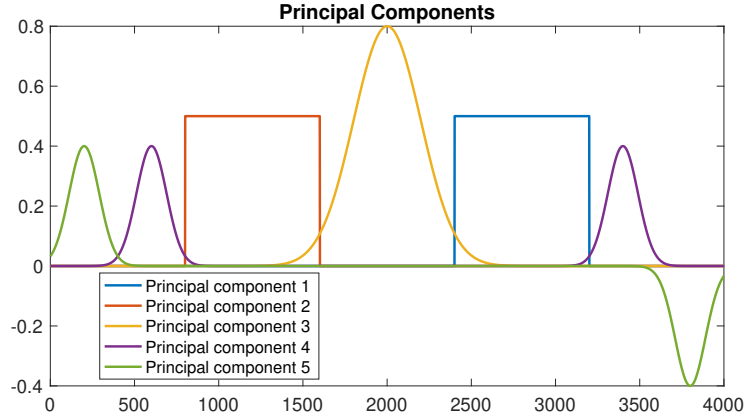}
\caption{
The five principal components used in the synthetic data.
}
\label{figPCs}
\end{figure}

\begin{table}[ht]
    \centering
    \caption{An average result of 5 random runs for random data with different settings of $(n,\mu)$. The subscript $k$ indicates a scale of $10^k$. iter-v denotes iterations to reach the optimal $\|v_k\|$ for the first time, for example, in the second row, 897 denotes the number of iterations that make $\|v_k\|$ firstly attain $10^{-8}$. iter-u denotes the number of using the new search direction $u_k$.}
    \begin{tabular}{c|c|cccccc}
    \hline$(n,\mu)$ & Algo & iter & iter-v & iter-u & $f$ & sparsity & $\|v_k\|$ \\
    \hline 
    (5000,1.5)   & ManPG  & 3000 & 897 & -    & $-4.59_1$ & 0.37 & $7.41_{-8}$ \\
    (5000,1.5)   & RPN-G   & 334 & - & 5  & $-4.59_1$  & 0.37 & $4.53_{-16}$ \\
    \hline
    (10000,1.8)   & ManPG  & 3000 & 1736 & -  & $-1.02_2$ & 0.32 & $2.19_{-8}$ \\
    (10000,1.8)   & RPN-G    & 580 &- & 6  & $-1.02_2$ & 0.32 & $5.69_{-16}$\\
    \hline
    (30000,2.0)   & ManPG  & 3000 & 1283 & -  & $-3.98_2$ & 0.22 & $1.19_{-8}$ \\
    (30000,2.0)   & RPN-G    & 347 & - & 5  & $-3.98_2$ & 0.22 & $5.25_{-15}$ \\
    \hline 
    (50000,2.2)   & ManPG  & 3000 & 1069  & -  & $-7.06_2$ & 0.18 & $4.56_{-7}$ \\
    (50000,2.2)   & RPN-G    & 789 & -  & 5   & $-7.06_2$ & 0.18 & $1.41_{-14}$ \\
    \hline 
    (80000,2.5)   & ManPG  & 3000 & 834 & -  & $-1.17_3$ & 0.17 & $1.41_{-6}$ \\
    (80000,2.5)   & RPN-G    & 839 & - & 6 & $-1.17_3$ & 0.17 & $1.94_{-15}$ \\
    \hline
    \end{tabular}
    \label{tab:1}
\end{table}

\begin{figure}
\centering 
\includegraphics[width=370pt]{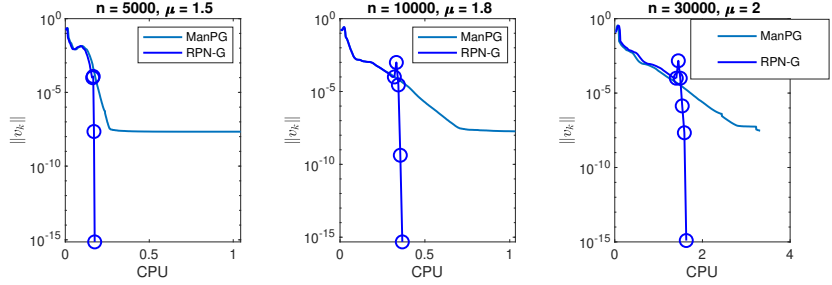}
\caption{Random data: the norm of search direction $v_k$ versus CPU for different $(n, \mu)$, where the blue circle indicates the use of the new direction $u_k$.}\label{fig:cpu} 
\end{figure}

\begin{figure}
\centering 
\includegraphics[width=370pt]{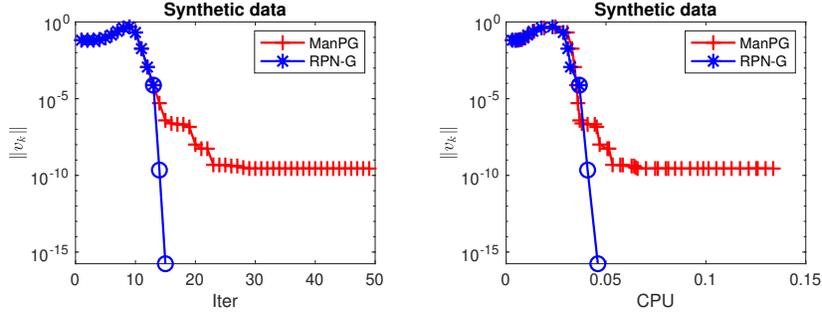}
\caption{Plots of $\|v_k\|$ versus iterations and CPU times respectively, where $\|v_k\|$ is the norm of search direction, data matrix $A\in \R^{4000\times 400}$ is from the synthetic data, $\mu$ is set to be 1.2. Note that the blue circle indicates the use of the new direction $u_k$.}\label{fig:sy_data} 
\end{figure}

\begin{remark} \label{flop}
Here, we discuss the numbers of floating point arithmetic (flop) for computing $v_k$ and $u_k$ per iteration. For sparse PCA with $r = 1$, the main computation cost is on $Ax$ and $A^{\T}Ax$ for ManPG, and its flops is $\mathcal{O}(4mn)$, where $A\in \R^{m\times n}$. Besides the computations in the function value and gradient evaluation, there are computations in the semismooth Newton iterations, where its dominant computational costs focus on  $\Psi (\lambda) = B_{x_k}^T \left( \prox_{t h} \bigl( x_k - t \left[\nabla f(x_k) + B_{x_k} \lambda \right] \bigr) - x_k \right)$ in~\eqref{eq:Psilambda} and 
\[
J_{\Psi}(\lambda_k)[d] = B_{x_k}^T \left(\partial \prox_{t h} (x_k - t \left[\nabla f(x_k) + B_{x_k} \lambda \right]) \right) \circ \left(-t B_{x_k} d \right),
\]
where $\circ$ denotes the entrywise product of two matrices. The computations of $\Psi(\lambda)$ and $J_{\Psi}(\lambda_k)[d]$ in one evaluation are respectively $\mathcal{O}(4n)$  and $\mathcal{O}(6n)$. Thus, the total computation cost in the semismooth method is on the order of $\mathcal{O}(ssn_k * 10n)$, where $ssn_k$ denotes the iteration number of semismooth Newton at $k$-th step. Therefore, the total complexity of ManPG in one evaluation is on the order of $\mathcal{O}(4mn + ssn_k * 10n)$.
    
For RPN and RPN-G, the direction $u_k$ is computed by additionally solving $J(x_k)[u_k] = - v_k$. We first compute $v_k$ by semismooth Newton method and its cost is also $\mathcal{O}(4mn + ssn_k * 10n)$. When $u_k$ is solved by the built-in Matlab function $cgs$, the dominant computational cost per iteration comes from multiplying a vector by the matrix $J(x_k)$. Assuming that $cgs$ takes $inner_k$ iterations for computing $u_k$, and the nonzero element of $x_k$ is $sp$, the computation of
\[
\begin{aligned}
    J(x_k)d =& -d + M_{x_k}d - M_{x_k} x_k (x_k^{\T} M_{x_k} x_k)^{-1} x_k^{\T} M_{x_k} d  \\
    & - t (M_{x_k} - M_{x_k} x_k (x_k^{\T} M_{x_k} x_k)^{-1} x_k^{\T} M_{x_k}) * (-2 A^{\T} A d + \lambda_k d)
\end{aligned}
\]
costs $\mathcal{O}( 2m(n+sp))$ flops. Thus, the overall computation of $u_k$ is $\mathcal{O}( inner_k * 2m(n+sp) + 4mn)$. Therefore, when the direction $u_k$ is computed, the total complexity of RPN and RPN-G in one evaluation is on the order of $\mathcal{O}(4mn + ssn_k * 10n + inner_k * 2m(n+sp))$. According to the numerical test, the inner iteration number for $cgs$ is usually 4 or 5 iterations.
\end{remark} 

\subsection{Comparisons with ManPG algorithm on the Stiefel manifold ($r > 1$)}\label{subsec:compare2}
\
In this section, we repeat the numerical experiments in Subsection~\ref{subsec:compare} for multiple values of $r$. The parameters are set as those in Subsection~\ref{subsec:compare}. The results are reported in Table~\ref{tab:stiefel}. The proposed method RPN-G is able to obtain higher accurate solutions in the sense that $\|v_k\|$ from RPN-G is in double precision whereas that from ManPG is only in single precision. In addition, the number of Riemannian proximal Newton steps is not large and usually is only 3-4 iterations. Since ManPG and RPN-G use different stopping criterion, their computational times are not reported in Table~\ref{tab:stiefel}. Instead, two typical runs for demonstrating the computational times are shown in Figure~\ref{fig:rand_syn_5}. We conclude that RPN-G converges faster in the sense of both computational time and the number of iterations.

\begin{table}[ht]
    \centering
    \caption{An average result of 5 random runs for random data with different settings of $(n,\mu, r)$. The subscript $k$ indicates a scale of $10^k$. iter-v denotes iterations to reach the optimal $\|v_k\|$ for the first time, for example, in the second row, 818 denotes the number of iterations that make $\|v_k\|$ firstly attain the order of $10^{-7}$. iter-u denotes the number of using the new search direction $u_k$.}
    \label{tab:stiefel}
    \begin{tabular}{cc|c|cccccc}
    \hline$\ \ \ (n,\mu)$\ \ \
    $r$ & Algo & iter & iter-v & iter-u & $f$ & sparsity & $\|v_k\|$ \\
    \hline 
    \ (200,0.6)\ 3   & ManPG  & 3000 & 818 & -    & $-7.93_0$ & 0.54 & $3.76_{-7}$ \\
    \ (200,0.6)\ 3   & RPN-G   & 612 & - & 3 & $-7.93_0$  & 0.54 & $8.09_{-15}$ \\
    \hline
      \ (300,0.8)\ 5   & ManPG  & 3000 & 384 & -    & $-9.16_0$ & 0.68 & $6.17_{-7}$ \\
    \ (300,0.8)\ 5   & RPN-G   & 245 & - & 4 & $-9.16_0$  & 0.68 & $2.71_{-14}$ \\
    \hline
    \ (500,0.6)\ 8   & ManPG  & 3000 & 1131 & -    & $-4.90_1$ & 0.45 & $1.03_{-7}$ \\
    \ (500,0.6)\ 8   & RPN-G   & 577 & - & 3  & $-4.90_1$  & 0.45 & $4.57_{-14}$ \\
    \hline
    \ \ (800,0.8)\ 10   & ManPG  & 3000 & 868 & -  & $-7.18_1$ & 0.50 & $3.11_{-7}$ \\
    \ \ (800,0.8)\ 10   & RPN-G    & 787 &- & 3  & $-7.18_1$ & 0.50 & $8.31_{-14}$\\
    \hline
    \end{tabular}
\end{table}

\begin{figure}
\centering 
\includegraphics[width=182pt]{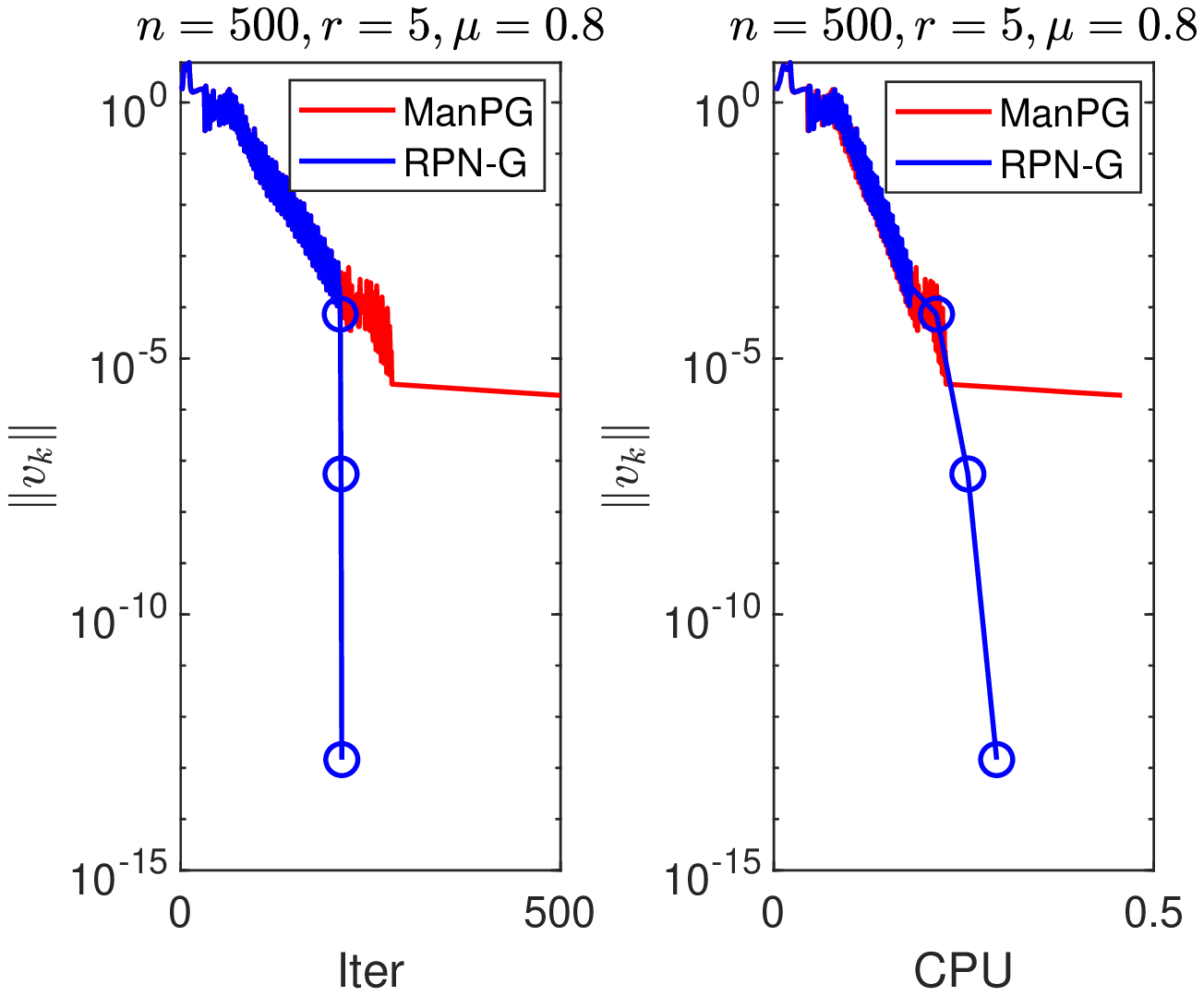}
\includegraphics[width=182pt]{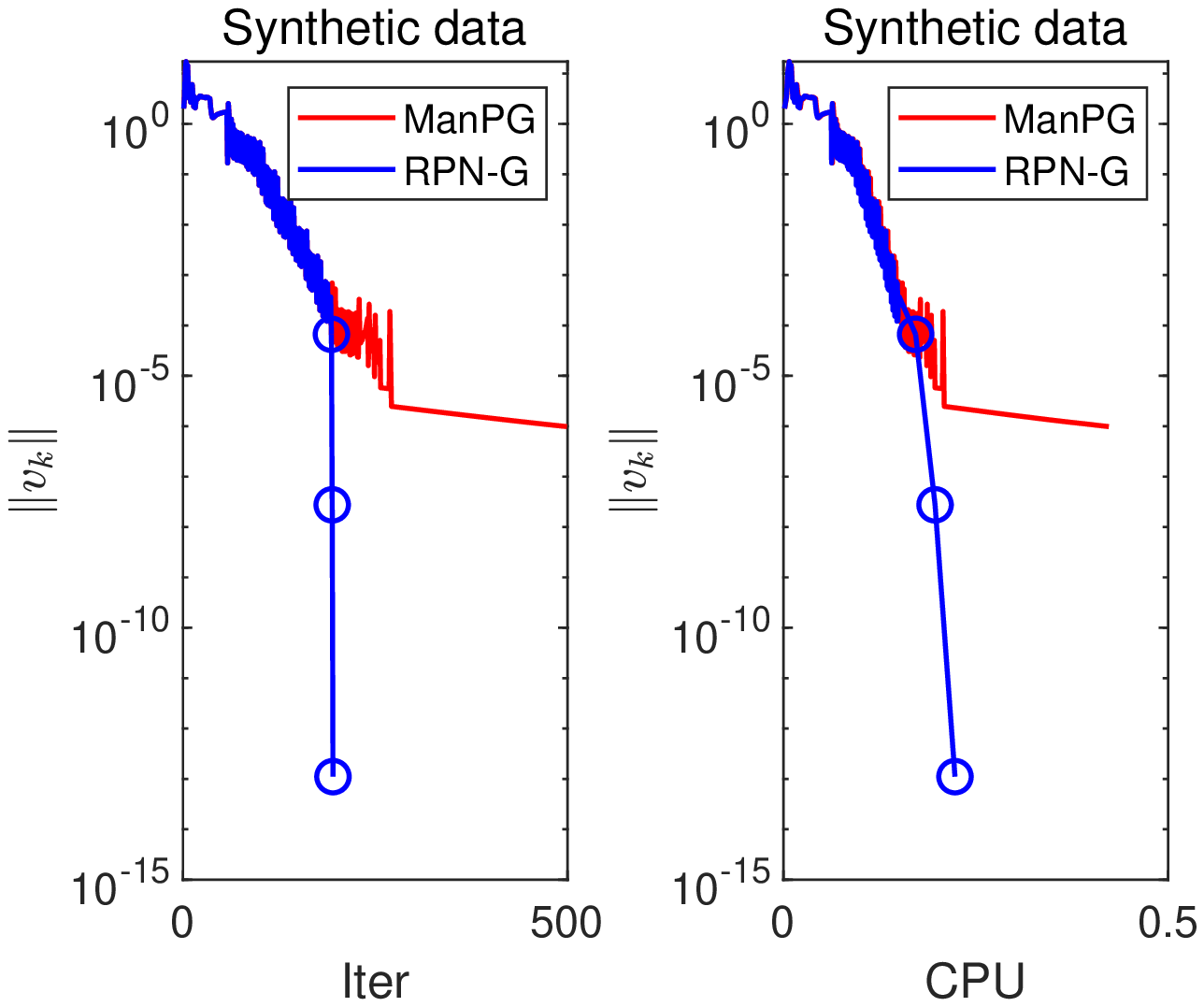}
\caption{Left: Random data: plots of $\|v_k\|$ versus iterations and CPU times respectively, where $(n,r,\mu) = (500,5,0.8)$. Right: Synthetic data: plots of $\|v_k\|$ versus iterations and CPU times respectively, where $\|v_k\|$ is the norm of search direction, data matrix $A\in \R^{400\times 40}$ is from the synthetic data, $\mu$ is set to be 0.8, $r$ is set to be 5. 
Note that the blue circle indicates the use of the new direction $u_k$.}\label{fig:rand_syn_5} 
\end{figure}

\section{Conclusion and Future Work}\label{sec:con} 

In this paper, we proposed a Riemannian proximal Newton method for solving optimization problems with separable structure, $f + \mu \|x\|_1$, over an embedded submanifold. It is proven that the proposed algorithm achieves quadratic convergence under certain reasonable assumptions. We further proposed a hybrid method that combines a Riemannian proximal gradient method and the Riemannian proximal Newton method. The hybrid method has been proven to have global convergence and the local quadratic convergence rate. Numerical results demonstrate the effectiveness of the proposed methods.

There are several directions for future research. In this paper, we have only considered a specific type of optimization problem, namely $f(x) + h(x)$ with $h(x) = \mu \|x\|_1$, on an embedded submanifold. It would be natural to extend the Riemannian proximal Newton method to handle more general nonsmooth functions $h(x)$ and more general manifolds. Additionally, we have only proposed a hybrid method for achieving global convergence, where the global convergence relies on the empirical selection of parameters. Therefore, it would be valuable to further investigate the globalization of the Riemannian proximal Newton method. The current analyses rely on solving Subproblem~\eqref{eq:subforv} and the linear system~\eqref{3-3} exactly. Proposing an inexact Riemannian proximal Newton method is also a valuable future research direction.

\section*{Acknowledgements}

The authors would like to thank Defeng Sun for the discussions on semismooth analysis.

\bibliographystyle{alpha}
\bibliography{rpn_refer}

\newpage
\appendix

\section{Semismooth Implicit Function Theorem}\label{appA}

The implicit function theorem for semismooth function comes from~\cite[Theorem 4]{Gowda2004}, it relies on the inverse function theorem for G-semismooth function~\cite[Theorem~2]{Gowda2004}, we first restate the result in Lemma~\ref{G:inv}.

\begin{lemma}[G-semismooth Inverse Function Theorem]\label{G:inv}
    Suppose that $f: \Omega \to \R^n$ be G-(strongly) semismooth function with respect to $\mathcal{K}_f$ on $\Omega$, where $\Omega\subset \R^n$ be an open set. Fix a point $x^0 \in \Omega$ and suppose that 
    \begin{itemize}
        \item[(i)] If $f$ is differentiable at $x\in \Omega$, then $\nabla f(x) \in \mathcal{K}_f (x)$.
        \item[(ii)] the multivalued mapping $x\mapsto \mathcal{K}_f (x)$ is compact valued and upper semicontinuous at each point of $\Omega$.
        \item[(iii)] $\mathcal{K}_f (x^0)$ consists of positively (negatively) oriented matrices.
        \item[(iv)] the topological index of $f$ at $x^0$ is 1 (respectively, -1), i.e., $\mathrm{index}(f, x^0)=1$(respectively, -1).
    \end{itemize}
    Then there exists an open neighborhood $\mathcal{U}$ of $x^0$ and an open neighborhood $\mathcal{V}$ of $f(x^0)$ such that $f: \mathcal{U} \to \mathcal{V}$ is bijective and $f$ has an inverse $h: \mathcal{V} \to \mathcal{U}$ that is locally Lipschitz, and $h$ is G-(strongly) semismooth at each $v \in \mathcal{V}$ with respect to $\mathcal{K}_h (v) = \{A^{-1}: A\in \mathcal{K}_f (u)\}$ and $v = f(u)$, where the map $v \mapsto \mathcal{K}_h(v)$ is compact valued and upper semicontinuous at each point of $\mathcal{V}$.
\end{lemma}

\begin{remark}
    In~\cite{Gowda2004}, the author mainly consider a class of H-differentiable function,
    and it follows from the definition of H-differentible~\cite[Definition~1]{Gowda2004},  G-(strongly) semismooth function belong to H-differentible functions, we refer interested reader to~\cite{Gowda2004} for more detail. .
\end{remark}

Note that Lemma~\ref{G:inv} requires the notion of positively (negatively) oriented matrices
and the notion of a topological index of a function.  A matrix is said to be positively (negatively) oriented if it has a positive (negative) determinant sign. The topological index is, however, not easy to verify. Fortunately, a sufficient condition has been given in~\cite[Theorem 6]{Pang2003} and we give it in Lemma~\ref{lemm-2-2}.

\begin{lemma} \label{lemm-2-2}
Let $\Phi : \mathbb{R}^n \to \mathbb{R}^n$ be Lipschitz continuous in an open neighborhood $\mathcal{D}$ of a vector $x^0$. Consider the following statements:
\begin{itemize}
    \item[(a)] every matrix in $\partial \Phi(x^0)$ is nonsingular;
    \item[(b)] for every $V \in \partial_{\mathrm{B}} \Phi(x^0)$, $\mathrm{sgn} \det V = \mathrm{index}(\Phi, x^0) = \pm 1$.
\end{itemize}
It holds that $(a) \Rightarrow (b)$. 
\end{lemma}

We present and prove the semismooth implicit function theorem in Corollary~\ref{lemm-2-3}, which is a  rearrangement  of Lemma~\ref{G:inv} and Lemma~\ref{lemm-2-2}.

\begin{corollary}[Restatement of Corollary~\ref{coro-2-1}]\label{lemm-2-3}
Suppose that $F: \mathbb{R}^n\times \mathbb{R}^m \to \mathbb{R}^m$ is a (strongly) semismooth function with respect to $\partial_{\mathrm{B}} F$ in an open neighborhood of $(x^0,y^0)$ with $F(x^0,y^0) = 0$. Let $H(y) = F(x^0, y)$. If every matrix in $\partial H(y^0)$ is nonsingular, then there exists an open set $\mathcal{V} \subset \mathbb{R}^n$ containing $x^0$, a set-valued function $\mathcal{K}:\mathcal{V} \rightrightarrows \mathbb{R}^{m \times n}$, and a G-(strongly) semismooth function $f:\mathcal{V} \to \mathbb{R}^m$ with respect to $\mathcal{K}$ 
satisfying $f(x^0) = y^0$, 
$$
F(x, f(x)) = 0
$$
for every $x\in \mathcal{V}$ and the set-valued function $\mathcal{K}$ is
\[
\mathcal{K}:x \mapsto \{-(A_y)^{-1} A_x: [A_x\ A_y] \in \partial_{\mathrm{B}} F \big(x,f(x)\big)\},
\]
where the map $x\mapsto \mathcal{K}(x)$ is compact valued and upper semicontinuous.

\end{corollary}

\begin{proof}
    Since every matrix in $\partial H (y^0)$ is nonsingular, where $H: \R^m \to \R^m : y \mapsto H(y) = F(x^0,y)$, then it follows Lemma~\ref{lemm-2-2}, for every $V \in \partial_{\mathrm{B}} H(x^0)$, $\mathrm{sgn} \det V = \mathrm{index}(H, x^0) = \pm 1$. Note that $H$ is also G-(strongly) semismooth at $y^0$ with respect to
    $$
    \mathcal{K}_H (y^0) = \{ A_y: [A_x\ A_y] \in \partial_{\mathrm{B}} F \big(x^0,y^0\big)\},
    $$
    where the map $y\mapsto \mathcal{K}_H (y)$ is compact valued and upper semicontinuous. Define $\mathcal{F}: \mathcal{D}\to \mathbb{R}^n\times \mathbb{R}^m$ by
    $$
    \mathcal{F}(x,y) = \left(x, F(x,y)\right),
    $$
    where $\mathcal{D} \subset \mathbb{R}^n\times \mathbb{R}^m$ is an open set including $(x^0,y^0)$. It is easily seen that $\mathcal{F}$ is G-(strongly) semismooth at $(x^0,y^0)\in \mathcal{D}$ with respect to 
    $$
    \mathcal{K}_{\mathcal{F}}(x^0,y^0) = \left\{ \begin{bmatrix}
        \I_n & 0\\
        A_x & A_y
    \end{bmatrix}: [A_x\ A_y] \in \partial_{\mathrm{B}} F \big(x^0,y^0\big)\right\},
    $$
    where the map $(x,y)\mapsto \mathcal{K}_{\mathcal{F}} (x, y)$ is compact valued and upper semicontinuous.
    
    We now verify conditions (i)-(iv) of Lemma~\ref{G:inv} for $\mathcal{F}$. 
   Since $\mathcal{K}_{\mathcal{F}}(x,y)$ is defined with the $\partial_{\mathrm{B}} F \big(x,y\big)$, then conditions (i) and (ii) are clearly satisfied for $\mathcal{F}$ and $\mathcal{K}_{\mathcal{F}}$. 
    As for conditions (iii) and (iv), it involves the property of topological index and the verification process of conditions (iii) and (iv) is exactly the same as that proof of~\cite[Theorem 4]{Gowda2004}, we 
    omit it and refer interested reader to~\cite{Gowda2004} for more detail. 
    
    According to Lemma~\ref{G:inv}, there exist a neighborhood $\mathcal{U}$ of $(x^0, y^0)$ and a neighborhood $\mathcal{V} \times \mathcal{W}$ of $\mathcal{F}(x^0,y^0)$ such that for every $v\in \mathcal{V}$ and $w\in \mathcal{W}$, there is a unique $z = (x,y) \in \mathcal{U}$ such that 
    $$
    \mathcal{F}(z) = (v, w).
    $$
    Since $\mathcal{F}(x^0, y^0) = (x^0, 0)$, we can fix $\omega$ to be 0. Thus, for each $v\in \mathcal{V}$, there is a unique $z = (x,y) \in \mathcal{U}$ with
    $$
    (x, F(x,y) ) = \mathcal{F}(z) = (v,0),
    $$
    which implies that $x = v$ and $F(x,y) = 0$. Thus, for every $x \in \mathcal{V}$, there is a unique $y$ such that $(x,y) \in \mathcal{U}$ and $F(x,y) = 0$. Therefore, $y$ is the function of $x$, we let $y = f(x)$, where $x\in \mathcal{V}$, then $\mathcal{F}(x, f(x)) = 0$ for any $x\in \mathcal{V}$. Furthermore, let $\mathcal{H}: \mathcal{V} \times \mathcal{W} \to \mathcal{U}$ denote the inverse of $\mathcal{F}$ on $\mathcal{U}$. According to Lemma~\ref{G:inv}, $\mathcal{H}$ is also G-(strongly) semismoooth at any $(v,w) \in \mathcal{V}\times \mathcal{W}$ with respect to $\mathcal{K}_{\mathcal{H}}(v,w) = \{M^{-1}: M\in \mathcal{K}_{\mathcal{F}}(x,y)\}$, where $\mathcal{F}(x,y) = (v,w)$ and the map $(v,w) \mapsto \mathcal{K}_{\mathcal{H}}(v,w)$ is compacted valued and upper semicontinuous at each point of $\mathcal{V}\times \mathcal{W}$. 
    
    Now $f$ can be written as
    $$
    f = \mathcal{P}_2 \circ \mathcal{H} \circ \phi,
    $$
    where $\phi$ is the inclusion map $\phi: \R^{n} \to \R^{n}\times \R^{m}:
    x \mapsto \phi(x) = (x,0)$ and $\mathcal{P}_2$ is the projection map $
    \mathcal{P}_2:\R^{n}\times \R^{m} \to \R^{m}: (x,y) \mapsto \mathcal{P}_2(x,y) = y$. Thus, $f$ is also G-(strongly) semismooth with respect to 
    $$
    \begin{aligned}
    \mathcal{K}(x) &= \left\{ \begin{bmatrix}
        0& \I_m
    \end{bmatrix} \begin{bmatrix}
        \I_n & 0\\ A_x & A_y
    \end{bmatrix}^{-1} \begin{bmatrix}
        \I_n \\0
    \end{bmatrix}  : [A_x\ A_y] \in \partial_{\mathrm{B}} F \big(x,f(x)\big).   \right\}\\
    &= \{-(A_y)^{-1} A_x: [A_x\ A_y] \in \partial_{\mathrm{B}} F \big(x,f(x)\big)\},
    \end{aligned}
    $$
    where the map $x\to \mathcal{K}(x)$ is compacted valued and upper semicontinuous. 
      
\end{proof}

\end{document}